\documentclass[aos]{imsart}

\RequirePackage{amsthm,amsmath,amsfonts,amssymb}
\RequirePackage[numbers]{natbib}
\usepackage{booktabs}
 \usepackage{algorithmic}
\usepackage{algorithm}

\startlocaldefs
\numberwithin{equation}{section}
\theoremstyle{plain}

\newtheorem{theorem}{Theorem}[section]
\newtheorem{lemma}[theorem]{Lemma}
\newtheorem{corollary}{Corollary}[section]
\newtheorem{proposition}{Proposition}[section]
\theoremstyle{remark}

\newtheorem{remark}{Remark}[section]
\newtheorem{assumption}{Assumption}[section]

\newcommand{\R}{\mathbb{R}}
\newcommand{\N}{\mathbb{N}}

\newcommand{\E}{\mathbb{E}}

\newcommand{\kdp}{{k^\delta_{\mathrm{dp}}}}
\newcommand{\mdp}{{m^\delta_{\mathrm{dp}}}}
\newcommand{\xdp}{x^\delta_{\kdp,\mdp}}
\newcommand{\omd}{\chi_{\Omega_{\delta/\rho}}}

\endlocaldefs

\begin{document}

\begin{frontmatter}
\title{Discretisation-adaptive regularisation of statistical inverse problems}
\runtitle{Discretisation-adaptive regularisation of statistical inverse problems}

\begin{aug}
\author[A]{\fnms{Tim}~\snm{Jahn}\ead[label=e1]{jahn@ins.uni-bonn.de}},
\address[A]{INS and HCM,
	University of Bonn\printead[presep={,\ }]{e1}}

\end{aug}

\begin{abstract}
We consider linear inverse problems under white noise. These types of problems can be tackled with, e.g., iterative regularisation methods and the main challenge is to determine a suitable stopping index for the iteration. Convergence results for popular adaptive methods to determine the stopping index often come along with restrictions, e.g. concerning the type of ill-posedness of the problem, the unknown solution or the error distribution. In the recent work \cite{jahn2021optimal} a modification of the discrepancy principle, one of the most widely used adaptive methods, applied to spectral cut-off regularisation was presented which provides excellent convergence properties in general settings. Here we investigate the performance of the modified discrepancy principle with other filter based regularisation methods and we hereby focus on the iterative Landweber method. We show that the method yields optimal convergence rates and present some numerical experiments confirming that it is also attractive in terms of computational complexity. The key idea is to incorporate and modify the discretisation dimension in an adaptive manner.\\
\end{abstract}



\end{frontmatter}

\section{Introduction}

Let $K:\mathcal{X}\to\mathcal{Y}$ be a compact operator with dense range between infinite-dimensional Hilbert spaces. We aim to solve the following equation

\begin{equation}\label{int:eq1}
Kx = y^\delta
\end{equation}

for $y^\delta = y^\dagger+\delta Z$ a noisy perturbation of the unknown true data $y^\dagger= Kx^\dagger$, with $x^\dagger=K^+y^\dagger$ the minimum norm solution (thus $K^+$ denotes the Moore-Penrose inverse of $K$). Here $Z$ is centered white noise with finite second moments, i.e., it holds that

\begin{itemize}
\item $\E[(Z,y)]=0$,
\item $\E[(Z,y)(Z,y')] = (y,y')$,
\item $(Z,y) \stackrel{d}{=} \frac{\|y\|}{\|y'\|}(Z,y')$
\end{itemize}

for all $y, y'\in\mathcal{Y}$, where $\delta>0$ denotes the noise level. This is the classical setting of a statistical inverse problem, see, e.g., the references Cavalier \cite{cavalier2011inverse}, Bissantz \textit{et al.} \cite{bissantz2007convergence} or O'Sullivan \cite{o1986statistical}. In general, inverse problems are ill-posed,  which means that its solution is unstable with respect to measurement noise.  Inverse problems are solved by regularisation, i.e., by replacing the unstable (pseudo)-inverse $K^+$ of $K$with a whole family of stable substitutes (indexed by a parameter, $k\in\N$ or $\alpha>0$). Prominent regularisation methods are Tikhonov regularisation, spectral cut-off and Landweber iteration. The concrete choice of the substitute, i.e., effectively the choice of the regularisation parameter $k$ or $\alpha$ is then done dependent on the noise level $\delta>0$ and eventually the measurement $y^\delta$and the overall goal is to obtain convergence (in a suitable sense) against the true solution $x^\dagger$ as the noise level $\delta$ tends to zero. Comprehensive convergence results for various regularisation methods are available for a priori parameter choice rules  depending only on the noise level $\delta$, see \cite{bissantz2007convergence}. However, these rules share a central drawback as they usually do not provide the optimal rate of convergence, unless one has very specific additional information on the unknown solution $x^\dagger$. From this drawback arises the need to study other parameter choice rules, which also take the measurement $y^\delta$ into account and which shall adapt automatically to the unknown properties of $x^\dagger$. For statistical inverse problems various adaptive parameter choice rules combined with different regularisation methods are known to be (almost) optimal. Here we name cross validation (Wahba \cite{wahba1977practical}), unbiased or penalized empirical risk minimization (Cavalier \textit{et al.}\cite{cavalier2002oracle,cavalier2006risk}) and the balancing principle (Math\'{e} and Pereverzyev \cite{mathe2006regularization}) due to Lepski \cite{lepskii1991problem}. Compared to popular methods from the deterministic setting, e.g., the discrepancy principle, these methods are computational substantially more expensive and may be not feasible in very high-dimensional settings due to the fact that usually a large set of estimators has to be computed and compared. The discrepancy principle is not directly applicable in the statistical setting, since it is based on the paradigm that the residual norm of the substitute should be approximately equal to the measurement error, which is undefined here due to unboundedness of white noise. In recent years different modifications of the discrepancy principle suited for the white noise setting have been studied. One approach is to presmooth \eqref{int:eq1} in order to obtain bounded noise (see Blanchard and Math\'{e} \cite{blanchard2012discrepancy} and Lu and Math\'{e} \cite{lu2014discrepancy}), while another is to formulate the principle directly on a discrete finite-dimensional statistical inverse problem in Blanchard, Hoffmann and Rei\ss{} \cite{ blanchard2018optimal, blanchard2018early} and Lucka \textit{et al.} \cite{lucka2018risk}.

 In this article we introduce a modification of the discrepancy principle which actively incorporates the choice of the discretisation size of the infinite-dimensional problem. The important role of the concrete discretisation is a classical topic of research for inverse problems, see, e.g., Math\'{e} and Pereverzyev \cite{mathe2001optimal, mathe2006regularization} or de Vito, Caponnetto and Rosasco \cite{de2006discretization}  for results in a statistical setting. Hereby, often the paradigm is that the discretisation is fixed and cannot be adapted. Compared to the above mentioned versions of the discrepancy principle the here proposed method turns out to be more robust, in the sense that it does not restrict to certain error distributions (e.g, Gaussian ones), operators (e.g., mildly ill-posed Hilbert-Schmidt operators) and smoothness classes of the true solution. Moreover, the numerical study at the end of this paper indicates that it might also be computationally more effective in some settings. The main reason for this is that the hyper parameter for the discrepancy principle can be chosen in a far less aggressive fashion than, e.g.,  in \cite{blanchard2018early, blanchard2018optimal}. 

The approach here is inspired by a recent modification of the discrepancy principle from Jahn \cite{jahn2021optimal, jahn2022probabilistic} for spectral cut-off, which is proven to have excellent theoretical properties. However, its applicability is restricted due to the fact that it needs the singular value decomposition. In this manuscript we apply the proposed ideas to the iterative Landweber method, extensions to other methods as , e.g., Tikhonov are possible, see Corollary \ref{cor4} below.

We consider a semi-discrete setting and introduce a family of operators $P_m:\mathcal{Y}\to U_m$ with $U_m\subset \mathcal{Y}$ finite-dimensional and apply the classical Landweber method to $P_mK$ and $P_my^\delta$, which yields the following recursive sequence  become 

\begin{equation}\label{int0:eq1}
x_{k+1,m}^\delta:= x^\delta_{k,m} - (P_mK)^*(P_mKx^\delta_{k,m}-P_my^\delta)
\end{equation}

with $k\ge 0$ and $x^\delta_{0,m}=0$. Note that we assume for simplicity that $\|P_mK\|\le 1$, the general case can be handled with introducing an additional step size parameter. The overall goal is to determine the stopping index $k$ and the discretisation dimension $m$, depending on the noise level $\delta$ and the measurement $y^\delta$. All theoretical results in this paper are restricted to the case where we measure along the singular vectors of $K$, i.e.,

\begin{equation}\label{int:eq3c}
P_m = \sum_{j=1}^m (u_j,\cdot) u_j
\end{equation}

is the projection onto the first $m$ singular vectors of $K$; the exact definition of the singular value decomposition is given in the next section. Obviously this is a serious restriction since one of the main advantages of Landweber iteration, compared to spectral cut-off, is that knowledge of the singular value decomposition is not needed. Still, the analysis of this case gives important insights, e.g., in the necessity of the \texttt{while}-loop in Algorithm 1 and 2 below. Moreover, it is useful for the case of general discretisations as we will explain now. Many inverse problems arising in practice are discretisations of integral equations, e.g., via a Galerkin scheme. These discretisations are reasonably chosen such that they converge to the infinite-dimensional problem as the discretisation dimension grows. In particular, the singular value decomposition of the discretised problems converge (see Babu\v{s}ka and Osborn \cite{babuska1991eigenvalue} and Harrach, Jahn and Potthast \cite{harrach2020regularising}). Therefore, increasing the discretisation dimension can be seen as adding more and more singular vectors (of high frequency). In an application the initial discretisation is often determined by the experimental design and cannot be changed easily. The starting point for a concrete application is therefore a fixed usually very high dimensional discretisation of the ideal infinite dimensional problem. However, in the setting of discretised integral equations it is often possible to approximately construct lower-dimensional discretisation directly from the initial discretisation, e.g., via averaging or dropping out rows of the matrix. In Section \ref{sec:5} we describe and test such a practical discretisation scheme based on averaging of an a priori given high-dimensional discretisation of an integral equation numerically. 

The parameter choice rule of choice in this manuscript is the discrepancy principle, which follows the paradigm that the reconstruction should explain the data up to the noise level, i.e., that the residual norm approximately equals the measurement noise. Through the discretisation the measurement error becomes finite (precisely it holds that $\E\|P_my^\delta-P_my^\dagger\|^2=\sum_{j=1}^m\E(y^\delta-y^\dagger,u_j)^2 = m\delta^2$) which allows to apply the classical discrepancy principle as follows. For a fudge parameter $\tau>1$ we set 

\begin{equation}\label{int0:eq2}
k^\delta_{\rm dp}(m):= \min\left\{k\in\N~:~ \|P_mKx^\delta_{k,m}-P_my^\delta\| \le \tau \sqrt{m}\delta \right\}.
\end{equation}

In order to determine the final reconstruction in \eqref{int0:eq1} we have to give both, a discretisation dimension $m$ and a stopping index $k$. For that we propose to first calculate \eqref{int0:eq2} for various $m$ and the select the final parameters by Algorithm 1.

 \begin{algorithm}\label{algorithm1}
	\caption{Modified discrepancy principle for Landweber iteration}
	\begin{algorithmic}[1]
		\STATE Given Landweber iterates $x_{\kdp(m),m}^\delta$ and $x_{2\kdp(m),m}^\delta$ with $\kdp(m)$ determined by the discrepancy principle , $m\in\N$;
		\STATE \textit{Initialisation}
		\STATE Set $n=0$ and $m^{n}=1$ and $I^n:=0$;
		\STATE \textit{Calculate ratio}
		\WHILE{ $I_n<\frac{1}{2\tau e^4}$}
		\STATE $n=n+1$;
        \STATE $m^n:=\min\left\{m\ge m^{n-1}~:~\kdp(m)\ge\max_{m'\ge m^{n-1}}\kdp(m')\right\}$;
		\STATE $I_n:=\frac{\left\|P_mKx^\delta_{2\kdp(m^n),m^n}-P_my^\delta\right\|}{\left\|P_mKx^\delta_{\kdp(m^n),m^n}-P_my^\delta\right\|}$;
		\ENDWHILE
		\STATE $x^\delta_{\kdp,\mdp}:=x^\delta_{\kdp(m^n),m^n}$;
	\end{algorithmic}
\end{algorithm}

Algorithm 1 follows the paradigm to maximise the stopping index, cf. Jahn \cite{jahn2021optimal, jahn2022probabilistic}.  This paradigm may feel counter-intuitive in the first place, since choosing a large iteration index potentially destabilises the solution. However, the discrepancy principle is known to be robust (note that $\tau>1$!) which is exploited here, see also \cite{lucka2018risk}. The while-loop hereby guarantees that the discretisation dimension is chosen sufficiently large and will be explained in more detail in later sections, where we also give a counter example showing that the loop is unavoidable. Algorithm 1 seems to be computationally costly in comparison to \cite{blanchard2018optimal}. Still, in a concrete application we would not have the luxury to add singular vectors gradually.  Typically, in a Galerkin approximation scheme one could double the number of nodes for each successive discretisation yielding the logarithm of the initial dimension as the number of discretisation levels and thus comparably few discretisation levels have to be compared with each other. In the end, due to the fact that the fudge parameter $\tau$ can be chosen quite large without deteriorating accuracy, only very few iterations have to be performed for large discretisation dimension. On the contrary, the early stopping discrepancy principle implements \eqref{int0:eq2} with $\tau=1$ for one fixed usually large dimension, which can be unstable (for not too large noise level $\delta$) and costly, because of the slow convergence of the Landweber recursion. On the other hand, using $\tau>1$ for a single fixed large dimension typically underestimates the optimal stopping index and thus produces less accurate results. See the numerical experiments in Section \ref{sec:5}.

We conclude the introduction with two corollaries avoiding the technicalities of the main result. First there holds convergence for arbitrary operator and noise. 

\begin{corollary}\label{cor1}
	Let $K$ be compact with dense range, $\|K\|\le 1$ and $x^\dagger=K^+y^\dagger$ with $y^\dagger\in\mathcal{R}(K)$. Then for $x^\delta_{\kdp,\mdp}$ determined by Algorithm 1 there holds, for any $\varepsilon>0$,
	
	$$\mathbb{P}\left(\|\xdp-x^\dagger\|\le \varepsilon\right)\to 1$$
	
	as $\delta\to0$,i.e., $\xdp$ converges to $x^\dagger$ in probability. 
\end{corollary}

Second we give some quantitative convergence rates. In inverse problems, such rates depend crucially of unknown abstract smoothness properties of the exact solution. Here we give the result for two classical settings, i.e., for polynomial ill-posed problems with H\"older source conditions and  for exponentially ill-posed problems with logarithmic source conditions. For these source conditions $\mathcal{X}_{\varphi, \rho}\subset \mathcal{X}$, which are subsets of the data space (see the next section for the exact definitions), we obtain minimax optimal rates.

\begin{corollary}\label{cor2}
	Let $K$ be either polynomially ill-posed or exponentially ill-posed (cf. \eqref{sec3:eq4} and \eqref{sec3:eq5}). Then, under H\"older \eqref{sec3:eq2} or logarithmic \eqref{sec3:eq3} source conditions respectively there exists $C>0$  such that
	
	$$\sup_{x^\dagger\in\mathcal{X}_{\varphi,\rho}}\mathbb{P}\left(\|\xdp-x^\dagger\|\le C \min_{m,k\in\N}\sup_{x\in\mathcal{X}_{\varphi,\rho}}\|x^\delta_{k,m}-x\|\right)\to 1$$
	
	as $\delta/\rho\to0$ ,i.e., with a probability converging to $1$ the minimax optimal convergence rate applies, where $x^\delta_{\kdp,\mdp}$ is determined by Algorithm 1.
	
\end{corollary}

In the next section we present the general results and explain the motivation for Algorithm 1 in more detail. The proofs are gathered in Section \ref{sec:4} and in the final Section \ref{sec:5} some numerical experiments for a practical version of the method are presented.

\section{Main results}\label{sec:2}

In order to give a more detailed motivation for our approach we first recap the results from \cite{jahn2021optimal}. Let $(\sigma_j,v_j,u_j)_{j\in\N}$ be the singular value decomposition, i.e., there holds $\sigma_1\ge \sigma_2\ge ... >0$ and $(u_j)_{j\in\N}$ and $(v_j)_{j\in\N}$ are orthonormal bases of $\mathcal{Y}$ and $\mathcal{N}(K)^\perp$ respectively which fulfill the  relations $K v_j=\sigma_j u_j$ and $K^*u_j=\sigma_j v_j$ for all $j\in\N$ with $K^*$ the adjoint of $K$. Using the singular value decomposition we define an approximation of the unknown $x^\dagger$ via spectral cut-off regularisation with regularisation parameter (also called truncation level in this context) $k\in\N$

$$x^\delta_{\rm{ SC},k}:= \sum_{j=1}^k \frac{(y^\delta,u_j)}{\sigma_j} v_j.$$

The truncation level $k$ has to be chosen dependent on the measurement $y^\delta$ and the noise level $\delta$. In order to apply the plain discrepancy principle one would demand that the norm of the residual equals the bound of the norm of the measurement error, i.e., that $k$ is defined by the relation

\begin{equation}\label{int:eq0}
\sqrt{\sum_{j=k+1}^\infty (y^\delta,u_j)^2} = \| Kx^\delta_{\rm{ SC},k}-y^\delta\| \approx \|y^\delta-y^\dagger\| = \sqrt{\sum_{j=1}^\infty (y^\delta-y^\dagger,u_j)^2}.
\end{equation}

However under white noise there holds $\E\|y^\delta-y^\dagger\|^2 = \delta^2 \sum_{j=1}^\infty(Z,u_j)^2 = \infty$ and consequently the choice \eqref{int:eq0} clearly is infeasible, as already mentioned above. As above we cut the sum using an additional discretisation parameter $m$ and obtain $\E[\sum_{j=1}^m(y^\delta-y^\dagger,u_j)^2] = m\delta^2$. We replace \eqref{int:eq0} with

\begin{equation}\label{int:eq2a}
\sum_{j=k+1}^m(y^\delta,u_j)^2\approx m \delta^2.
\end{equation}

The 'solution' of \eqref{int:eq2a} now depends on the discretisation level $m$ and the main result of \cite{jahn2021optimal}  states that maximising over $m$ yields a truncation level providing optimal convergence against the unknown solution in probability; in form of an oracle inequality.
This result motivates to determine the stopping index $k$ for the discretised Landweber iteration \eqref{int0:eq1} as 

\begin{align}\label{int:pc}
\bar{k}_{dp}^\delta:=\max_{m\in\N} \kdp(m);
\end{align}

with $\kdp(m)$ from \eqref{int0:eq2} and the corresponding discretisation dimension

\begin{align}\label{int:pc1}
\bar{m}_{dp}^\delta:=\arg\max_{m\in\N} \kdp.
\end{align}

As said above, somewhat surprisingly however this approach does not yield convergence. 

\begin{theorem}\label{th1}
	Let $K$ be compact with dense range and $\sigma_1=\sigma_2<1$ and $Z$ be Gaussian white noise. Moreover, let $x^\dagger = v_1+v_2/2$ and $\tau\ge2$. Then for the choice \eqref{int:pc} and \eqref{int:pc1} there holds
	
	$$\limsup_{\delta\to0}\mathbb{P}\left(\|x^\delta_{\bar{k}^\delta_{\rm dp},\bar{m}^\delta_{\rm dp}} - x^\dagger\|>1/\sqrt{2}\right)\ge \frac{1}{48}$$
	
 In particular, $x^\delta_{\bar{k}^\delta_{\rm dp},\bar{m}^\delta_{\rm dp}}$ does not converge to $x^\dagger$.
\end{theorem}

In the proof of Theorem \ref{th1} it becomes clear that the breakdown of the convergence is due to fatal early stopping. To overcome this early stopping we suggest the following adaptive choice: Set $m^0:=1$ and recursively

\begin{align*}
m^{n}:=\min\left\{ m \ge m^{n-1}~:~ k(m) \ge \max_{m'\ge m^{n-1}} k(m')\right\}.
\end{align*}

As a stopping rule of the recursion we then use

\begin{align}\label{int:eq2}
n_*:=& \min\left\{n\in\N~:~\frac{\|Kx^\delta_{2\kdp(m^n),m^n} - P_{m^n}y^\delta\|}{\|Kx^\delta_{\kdp(m^n),m^n} - P_{m^n}y^\delta\|} \ge \frac{1}{2\tau e^4}\right\},
\end{align}

and our final choice for stopping index and discretisation level becomes

\begin{align*}
m^\delta_{dp}:=&m^{n_*}\\\notag
k^\delta_{dp}:=&k_{dp}^\delta(m^\delta_{dp}),
\end{align*}

see Algorithm 1 for the implementation. The above choice still has the same spirit - maximising the stopping index. However it is taken care of that also the discretisation dimension is sufficiently large, for what the expression inside the curled brackets in \eqref{int:eq2} can be seen as a test for, see \eqref{th2:eq00} in the proof of Theorem \ref{th2} below. We remark that the threshold $\frac{1}{2\tau e^4}$ is not optimal and usually can be replaced with $\frac{1-\varepsilon}{\tau}$, see Corollary \ref{cor3} below.

 In order to formulate the main result we start by giving the rigorous definition of the source conditions. A function $\varphi:[0,\|K\|^2]\to\R$ is called an index function, if it is continuous monotonically increasing and fulfills $\varphi(0)=0$. A source condition (see Hofmann and Math\'{e} \cite{hofmann2007analysis}) is then defined as

\begin{equation}\label{sec3:eq1}
\mathcal{X}_{\varphi,\rho}:=\left\{x=\varphi(K^*K)\xi = \sum_{j=1}^\infty \varphi(\sigma_j^2)(\xi,v_j)v_j~|~\xi\in\mathcal{X},~\|\xi\|\le \rho\right\}.
\end{equation}

The convergence rate under a source condition will be expressed using an  auxiliary function. We first define for $ x\in[1,\infty)$ and $\alpha\in(0,\|K\|^2]$

\begin{align*}
\sigma_x^2:&=\sigma_{\lfloor x \rfloor}^2(\lfloor x \rfloor +1-x)\sigma_{\lfloor x\rfloor +1}^2(x-\lfloor x \rfloor),\\
x_{\alpha}:&=\min\{ x \ge 1 ~:~ \sigma_x^2 = \alpha\}.
\end{align*}

Note that $\sigma_{x_{\alpha}}^2 = \alpha$ and set

\begin{equation}
\Theta:(0,\|K\|^2),\alpha\mapsto \frac{\alpha \varphi^2(\alpha)}{x_\alpha}.
\end{equation}

Obviously, $\Theta$ is continuous and strictly monotonically increasing with $\lim_{\alpha\to 0}\Theta(\alpha)=0$. Therefore it is invertible on $(0,\infty)$. 

\begin{remark}\label{rem1}
	We give a short motivation of the function $\Theta$ and its relation to optimal rates. One can show that there exists $c>0$ such that
	
	\begin{align*}
	&\min_{k,m\in\N}\sup_{\xi\in\mathcal{X}_{\varphi,\rho}}\E\|x_{k,m}^\delta-x^\dagger\|^2\\
	= &\min_{k,m\in\N}\sup_{\xi\in\mathcal{X}_{\varphi,\rho}}\left\{\delta^2 \sum_{j=1}^m\frac{(1-(1-\sigma_j^2)^k)^2}{\sigma_j^2} + \sum_{j=1}^m(1-\sigma_j^2)^{2k}\varphi^2(\sigma_j^2)(\xi,v_j)^2\right.\\
	&\qquad\qquad\qquad\qquad \left. + \sum_{j=m+1}^\infty\varphi^2(\sigma_j^2)(\xi,v_j)^2\right\}\\
	\ge &c \min_{l\in\N} \left\{\delta^2 \sum_{j=1}^l\frac{1}{\sigma_j^2} + \varphi^2(\sigma_{l+1}^2)\rho^2\right\}\\
	&\ge \frac{c}{2}\min\left(\delta^2 \sum_{j=1}^{l_*} \frac{1}{\sigma_j^2} + \varphi^2(\sigma_{l_*+1}^2) \rho^2,\delta^2 \sum_{j=1}^{l_*-1} \frac{1}{\sigma_j^2} + \varphi^2(\sigma_{l_*}^2) \rho^2\right),
	\end{align*}
	
	where $l_*:=\min\left\{l\in\N~:~ \delta^2\sum_{j=1}^l\frac{1}{\sigma_j^2} \ge \varphi^2(\sigma_{l+1}^2)\rho^2\right\}$ is the index balancing the both terms in the sum. Now there holds

	$$\delta^2 \frac{l}{\sigma_l^2} \ge \delta^2\sum_{j=1}^{l}\frac{1}{\sigma_j^2}\quad \mbox{and} \quad \varphi^2(\sigma_{l+1}^2)\rho^2\le \varphi^2(\sigma_l^2)\rho^2.$$
	
	Set $\alpha_*=\Theta^{-1}\left(\frac{\delta^2}{\rho^2}\right)$. By definition it then holds that
	
	\begin{equation}\label{rem1:eq1}
        	\varphi^2(\sigma_{x_{\alpha_*}}^2)\rho^2 = \varphi^2(\alpha_*)\rho^2 = \Theta(\alpha_*) \frac{x_{\alpha_*}}{\alpha_*}\rho^2  = \frac{x_{\alpha_*}}{\sigma_{x_{\alpha_*}}^2}\delta^2,
	\end{equation}

	i.e., $x_{\alpha_*}$ is an approximation for the balancing index $l_*$ and $\varphi^2\left(\Theta^{-1}\left(\frac{\delta^2}{\rho^2}\right)\right)\rho^2$ is an approximation for the optimal convergence rate.
\end{remark}

Moreover we need the notation of a qualification of a regularisation method. An index function $\varphi$ is a qualification of Landweber iteration if

\begin{equation}\label{sec3:eq00}
\sup_{0<\lambda\le \|K\|^2}(1-\lambda)^{k}\varphi(\lambda)\le \varphi\left(\frac{1}{k}\right).
\end{equation}

We will formulate our main result for all source conditions with the following property.

\begin{assumption}\label{as1}
	$\varphi$ is an index function such that $\varphi^2$ is invertible and the function $\phi:[0,\varphi(\|K\|^2)^2]\to\R,~\lambda\mapsto \lambda \cdot \left(\varphi \cdot \varphi\right)^{-1}(\lambda)$ is convex. Moreover, $\varphi$ and the function $\lambda\mapsto \sqrt{\lambda}\varphi(\lambda)$ are qualifications of the Landweber method in the sense of \eqref{sec3:eq00}.
\end{assumption}

Note that actually Assumption \ref{as1} is a standard assumption for convergence analysis and no real restriction, since for any $x^\dagger$ there exists an index function $\varphi$ and $\rho>0$ with $x^\dagger\in\mathcal{X}_{\varphi,\rho}$ fulfilling the requirements of the assumption.  Popular concrete source conditions are H\"older source conditions 

\begin{equation}\label{sec3:eq2}
\varphi(t)=t^\frac{\nu}{2}
\end{equation}

and logarithmic source conditions

\begin{equation}\label{sec3:eq3}
\varphi(t)=\left(-\log(t)\right)^{-\frac{p}{2}},\end{equation}

which are naturally considered with polynomially or exponentially ill-posed problems respectively, i.e.,

\begin{equation}\label{sec3:eq4}
\sigma_j^2 \asymp j^{-q}
\end{equation}

and 

\begin{equation}\label{sec3:eq5}
\sigma_j^2 \asymp \exp(-aj),
\end{equation}

where $\nu,p,a,q>0$. It is easy to check that \eqref{sec3:eq2} and \eqref{sec3:eq3} fulfill Assumption \ref{as1}. We now formulate our main result about the rate of convergence for general $K$  under general source conditions.

\begin{theorem}\label{th2}
	Let $K$ be compact with dense range, $\|K\|\le 1$ and assume that $\varphi$ fulfills Assumption \ref{as1}. Then there exists $L>0$ such that for $\xdp$ determined by Algorithm 1  there holds
	
	$$\mathbb{P}\left(\sup_{x^\dagger\in\mathcal{X}_{\varphi,\rho}}\|\xdp-x^\dagger\|\le L  \rho \varphi\left(\Theta^{-1}\left(\frac{\delta^2}{\rho^2}\right)\right)\right)\to 1 $$
	
	as $\delta/\rho\to 0$.
\end{theorem}
 Note that contrarily to the setting in Corollary \ref{cor2} it does not hold that the rate is minimax, due to the generality of the operator $K$.

In some cases, e.g., for exponentially ill-posed problems under logarithmic source condition one would like to guarantee asymptotically optimal convergence, e.g., $L=1+o(1)$. We just mention that the approach presented here will not yield this type of convergence, another modification would be required to achieve this.

\subsection{Additional remarks}
The first corollary shows that $\frac{1}{2e^4\tau}$ in Algorithm 1 $\eta$ can be replaced with  $(1-\varepsilon)/\tau$ without deteriorating the convergence rate asymptotically in the setting of Corollary \ref{cor2}.

\begin{corollary}\label{cor3}
	Assume that  $K$ is either polynomially ill-posed or exponentially ill-posed (cf. \eqref{sec3:eq4} and \eqref{sec3:eq5}). Then, under H\"older \eqref{sec3:eq2} or logarithmic \eqref{sec3:eq3} source conditions respectively, for any $\varepsilon>0$ there exists $C_\tau'>0$  such that
	
	$$\sup_{x^\dagger\in\mathcal{X}_{\varphi,\rho}}\mathbb{P}\left(\|\xdp-x^\dagger\|\le C_\tau' \min_{m,k\in\N}\sup_{x\in\mathcal{X}_{\varphi,\rho}}\|x^\delta_{k,m}-x\|\right)\to 1$$
	
	as $\delta/\rho\to0$ where $x^\dagger_{\kdp,\mdp}$ is determined with Algorithm 1 with $\frac{1}{2e^4\tau}$ replaced by $(1-\varepsilon)/\tau$.
\end{corollary}

We quickly discuss extensions to other regularisation method, e.g., the Tikhonov regularisation. The Tikhonov regularisation is defined as the solution of the following minimisation problem

\begin{equation}\label{subsec:eq1}
\min_{x\in\mathcal{X}}\|Kx-y^\delta\|^2 + \alpha\|x\|^2,
\end{equation}

where the regularisation parameter $\alpha>0$ balances the both terms and its role is comparable to that of $k^{-1}$ for Landweber iteration thus a smaller $\alpha$ yields a stronger regularisation. As before, in order to obtain well-defined expressions we have to discretise. It is not hard to show that the solution of the discretised version of \eqref{subsec:eq1} is 

\begin{equation}
x_{\alpha,m}^\delta:=\left((P_mK)^*P_mK+\alpha {\rm Id}\right)^{-1}(P_mK)^* P_m y^\delta=\sum_{j=1}^m\frac{\sigma_j}{\alpha+\sigma_j^2}(y^\delta,u_j)v_j,
\end{equation}

under the discretisation scheme \eqref{int:eq3c}, where ${\rm Id}$ denotes the identity on $\mathcal{X}$. The discrepancy principle for Tikhonov regularisation is typically implemented as

$$\alpha^\delta_{\rm dp}(m):=q^j\quad\mbox{with}\quad j=\min\left\{ i\ge 0~:~ \|P_m K x^\delta_{q^i,m} - P_my^\delta\|\le \tau \sqrt{m}\delta\right\},$$

with $q\in(0,1)$ fixed. Most of the results carry over to that case, with the obvious modification that the general source condition $\varphi$ and $\lambda\mapsto \sqrt{\lambda} \varphi(\lambda)$ have to be qualifications of Tikhonov regularisation, where an index function $\varphi$ is called qualification of Tikhonov's regularisation, if

\begin{equation}\label{subsec:eq1a}
\sup_{0<\lambda\le \|K\|^2}\frac{\alpha}{\alpha+\lambda}\varphi(\lambda)\le \varphi\left(\alpha\right).
\end{equation}

 The most interesting part is to modify the \texttt{while}-loop in Algorithm 1, which tests, whether the discretisation dimension is large enough. There are different approaches to do so and we will present one option here in Algorithm 2, and additional assume that the source condition $\varphi$ is concave. Note that this is not a serious restriction, see the proof of Corollary \ref{cor1} below. 

 \begin{algorithm}\label{algorithm2}
	\caption{Modified discrepancy principle for Tikhonov regularisation}
	\begin{algorithmic}[1]
		\STATE Given Tikhonov reconstructions $x_{\alpha^\delta_{\rm dp}(m),m}^\delta$ and $x_{t\alpha^\delta_{\rm dp}(m),m}^\delta$ with $\alpha^\delta_{\rm dp}(m)$ determined by the discrepancy principle , $m\in\N$, $t=1/8\tau$;
		\STATE \textit{Initialisation}
		\STATE Set $n=0$ and $m^{n}=1$ and $I^n:=0$;
		\STATE \textit{Calculate ratio}
		\WHILE{ $I_n<\frac{1}{8\tau}$}
		\STATE $n=n+1$;
		\STATE $m^n:=\min\left\{m\ge m^{n-1}~:~\alpha^\delta_{\rm dp}(m)\le\min_{m'\ge m^{n-1}}\alpha^\delta_{\rm dp}(m)\right\}$;
		\STATE $I_n:=\frac{\left\|P_mKx^\delta_{t\alpha^\delta_{\rm dp}(m),m^n}-P_my^\delta\right\|}{\left\|P_mKx^\delta_{\alpha^\delta_{\rm dp}(m),m^n}-P_my^\delta\right\|}$;
		\ENDWHILE
		\STATE $x^\delta_{\alpha^\delta_{\rm dp},\mdp}:=x^\delta_{\alpha^\delta_{\rm dp}(m^n),m^n}$;
	\end{algorithmic}
\end{algorithm}

The result is similar to Theorem \ref{th2}.

\begin{corollary}\label{cor4}
	Let $K$ be compact with dense range and	assume that $\varphi$ fulfills Assumption \ref{as1} and $\varphi$ and $\lambda\mapsto \sqrt{\lambda} \varphi(\lambda)$ are qualification of Tikhonov regularisation in the sense of \eqref{subsec:eq1a}. Then there exists $L'>0$ such that for $x^\delta_{\alpha^\delta_{\rm dp},\mdp}$ determined by Algorithm 2 there holds
	
	$$\mathbb{P}\left(\sup_{x^\dagger\in\mathcal{X}_{\varphi,\rho}}\|x^\delta_{\alpha^\delta_{\rm dp},\mdp}-x^\dagger\|\le L'  \rho \varphi\left(\Theta^{-1}\left(\frac{\delta^2}{\rho^2}\right)\right)\right)\to 1 $$
	
	as $\delta/\rho\to 0$.
\end{corollary}

\section{Proofs}\label{sec:4}
In this section we give the proofs of the main results Theorem \ref{th1} and \ref{th2} and of the Corollaries. For the proofs we rely on the following central proposition, which states that the measurement error is  highly concentrated simultaneously for all $m$ large enough. Since we aim to prove bounds in probability this allows to focus on cases where the error behaves nicely.

\begin{proposition}\label{prop0}
	For any $\epsilon>0$ and $\kappa\in\N$ there holds
	
	$$\mathbb{P}\left(\left|\sum_{j=1}^m(y^\delta-y^\dagger,u_j)^2-m\delta^2\right| \ge \varepsilon m\delta^2,\quad\forall m\ge \kappa\right) \le\frac{1}{\varepsilon} \E\left[\left|\frac{1}{\kappa}\sum_{j=1}^\kappa\left((Z,u_j)^2-1\right) \right|\right] \to 0$$
	
	as $\kappa\to\infty$.
\end{proposition}

The proof of Proposition \ref{prop0} is based on the backward martingale property of the residual error and can be found in \cite{jahn2021optimal} (Proposition 4.1). 
Moreover, we need the following well-known properties of Landweber iteration, which can be proven by standard means (differentiation with respect to $\lambda$). For all $k\in\N$ there holds

\begin{align}\label{pr:eq0}
   \sup_{0<\lambda\le \|K\|^2}(1-\lambda)^{2k} \lambda &\le \frac{1}{k}.
\end{align}

and

\begin{equation}\label{pr:eq00}
\sup_{0<\lambda\le \|K\|^2}\frac{(1-(1-\lambda)^k)^2}{\lambda}  \le k.
\end{equation}

We start with the proof of the main result.

\subsection{Proof of Theorem \ref{th2}}

So $x^\dagger=\varphi(K^*K)\xi$ with $\xi\in\mathcal{X}$ and $\|\xi\|\le \rho$. We first define $\alpha_*=\alpha_*(\delta/\rho):=\Theta^{-1}\left(\delta^2/\rho^2\right)$  and $m_*:=\lceil x_{\alpha_*}\rceil$ with $\Theta$ and $x_{\cdot}$ given in the preceding section. The number $m_*$ can be seen as the optimal discretisation level, see Remark \ref{rem1} above. While the singular value decomposition is not needed to calculate the Landweber iterates, it is handy for the analysis to express them in terms of the singular vectors. The expansion is given by

\begin{align*}
x^\delta_{k,m} = \sum_{j=1}^m \frac{1-(1-\sigma_j^2)^{k}}{\sigma_j}(y^\delta,u_j)v_j.
\end{align*}

By \eqref{rem1:eq1} and monotonicity of $\Theta$ it follows that

\begin{equation}\label{th2:eq00}
\sigma_{m}^2\varphi^2(\sigma_{m}^2)\rho^2\ge m\delta^2\qquad\mbox{and}\qquad \sigma_{m'}^2\varphi^2(\sigma_{m'}^2)\rho^2 \le m'\delta^2
\end{equation}

for all $m\le m_*$ and $m'\ge m_*+1$.  We will perform the analysis on a set of nice events

\begin{equation}\label{proof:eq1}
\Omega_{\delta/\rho}:=\left\{\left|\sqrt{\sum_{j=1}^m(y^\delta-y^\dagger,u_j)^2}-\sqrt{m}\delta\right|\le \min\left(\frac{\tau-1}{2},\frac{1}{10}\right)\sqrt{m}\delta,~\forall m\ge m_*\right\}
\end{equation}

where the error behaves regularly. Clearly, $\mathbb{P}\left(\Omega_{\delta/\rho}\right)\to 1$ as $\delta/\rho\to0$ by Proposition \ref{prop0}. We will frequently use that by definition of $\Omega_{\delta/\rho}$ there holds

\begin{align*}
\sqrt{\sum_{j=1}^m(y^\delta-y^\dagger,u_j)^2}\omd&\le \frac{\tau+1}{2}\sqrt{m}\delta,\\
\frac{9}{10}\sqrt{m}\delta\omd\le\sqrt{\sum_{j=1}^m(y^\delta-y^\dagger,u_j)^2}\omd&\le \frac{11}{10}\sqrt{m}\delta
\end{align*}
for all $m\ge m_*$.  As a first step we show boundedness of $\mdp$.

\begin{lemma}\label{lem1}
	There holds
	
	$$\mdp \omd \le B_\tau m_*\omd$$
	
	for $B_\tau:=A_\tau\max\left(\frac{2}{\tau-1},\left(\frac{21}{4}\right)^2\right)$ and $A_\tau$ given below.	
\end{lemma} 
\begin{proof}[Proof of Lemma \ref{lem1}]
	 We need to show that the Algorithm is well-defined and we first need to control the $m^n$ for $n\in\N$. In order to do so we claim that 
	
	\begin{equation}\label{lem1:eq1}
	k_{dp}^\delta(m)\omd \le k_{dp}^\delta(m')\omd
	\end{equation}
	
	for all pairs $(m,m')\in\N^2$ with $m\ge A_\tau m'$ and $m'\ge m_*$, where $A_\tau:=2\left(\frac{2}{\tau-1}+\frac{3\tau+1}{\tau-1}\right)^2$. To prove \eqref{lem1:eq1} we argue by contradiction and assume that there exist $m,m'$ with $m\ge A_\tau m'$ and $m'\ge m_*$ such that $k_{dp}^\delta(m) > k_{dp}^\delta(m')$. Then, for $k=k_{dp}^\delta(m')$ we apply the triangle inequality to the defining relation of the discrepancy principle (noticing that $k<\kdp(m)$) and obtain
	
	\begin{align}\label{pr:eq1}
	\tau\sqrt{m}\delta &< \sqrt{\sum_{j=1}^m(1-\sigma_j^{2})^{2k}(y^\delta,u_j)^2}\\\notag
	&\le \sqrt{\sum_{j=1}^m(1-\sigma_j^2)^{2k}(y^\delta-y^\dagger,u_j)^2}+\sqrt{\sum_{j=1}^m(1-\sigma_j^2)^{2k}(y^\dagger,u_j)^2}\\\notag
	 &\le \sqrt{\sum_{j=1}^m(y^\delta-y^\dagger,u_j)^2} +\sqrt{\sum_{j=1}^m(1-\sigma_j^2)^{2k}(y^\dagger,u_j)^2}.
	 \end{align}
	 
	 Similarly, with the reverse triangle inequality we deduce
	 
	 \begin{align}\label{pr:eq2}
	\tau\sqrt{m'}\delta &\ge \sqrt{\sum_{j=1}^{m'}(1-\sigma_j^{2})^{2k}(y^\delta,u_j)^2}\\\notag
	&\ge\sqrt{\sum_{j=1}^{m'}(1-\sigma_j^2)^{2k}(y^\dagger,u_j)^2}  - \sqrt{\sum_{j=1}^{m'}(y^\delta-y^\dagger,u_j)^2}.
	\end{align}
	
	By definition of $\omd$, from \eqref{pr:eq1} and \eqref{pr:eq2}  we then obtain
	
	\begin{align*}
	\frac{\tau-1}{2}\sqrt{m}\delta\omd <\sqrt{\sum_{j=1}^m(1-\sigma_j^2)^{2k}(y^\dagger,u_j)^2}
	\end{align*}
	and
	\begin{align*}
        \frac{3\tau+1}{2}\sqrt{m'}\delta\omd \ge \sqrt{\sum_{j=1}^{m'}(1-\sigma_j^2)^{2k}(y^\dagger,u_j)^2}\omd.
	\end{align*}
	
	Subtracting the second from the first inequality and concavity of the square root yields 
	
	\begin{align}\label{prop1a:eq1}
	&\frac{3\tau+1}{2}\left(\frac{\tau-1}{3\tau+1} \sqrt{m}-\sqrt{m'}\right)\delta\omd\\\notag
	 < &\sqrt{\sum_{j=1}^m(1-\sigma_j^2)^{2k}(y^\dagger,u_j)^2}-\sqrt{\sum_{j=1}^{m'}(1-\sigma_j^2)^{2k}(y^\dagger,u_j)^2}\\\notag
	\le & \sqrt{\sum_{j=m'+1}^\infty(1-\sigma_j^2)^{2k}(y^\dagger,u_j)^2}
	\le \sqrt{\sum_{j=m'+1}^\infty(y^\dagger,u_j)^2}\\\notag
	= &\sqrt{\sum_{j=m'+1}^\infty\sigma_j^2\varphi^2(\sigma_j^2) (\xi,v_j)^2}\le \sqrt{\sup_{0<\lambda\le \sigma_{m'+1}^2}\lambda\varphi^2(\lambda)}\sqrt{\sum_{j=m'+1}^\infty(\xi,v_j)^2}\\\notag
	\le &\sqrt{\sigma_{m'+1}^2\varphi^2\left(\sigma_{m'+1}^2\right)\rho^2}
	 \le \sqrt{m'+1} \delta\le \sqrt{2m'}\delta,
	\end{align}
	
	where we used $m'+1\ge m_*+1$ and \eqref{th2:eq00} in the last line. On the other hand,
	
	\begin{align}\label{prop1a:eq2}
	\frac{3\tau+1}{2}\left(\frac{\tau-1}{3\tau+1} \sqrt{m}-\sqrt{m'}\right)&\ge \frac{3\tau+1}{2}\left(\frac{\tau-1}{3\tau+1} \sqrt{A_\tau} -1 \right)\sqrt{m'} = \sqrt{2m'}
	\end{align}
	
	Using the above inequalities \eqref{prop1a:eq1} and \eqref{prop1a:eq2} gives a contradiction
	
	$$\sqrt{m'} \delta\omd<\sqrt{m'}\delta,$$

	therefore the claim \eqref{lem1:eq1} is proven. As a second ingredient we state that it holds that
	
	\begin{equation}\label{lem1:eq2}
	\frac{\|Kx^\delta_{2k^\delta_{\rm dp}(m), m} - P_my^\delta\|}{\|Kx^\delta_{k^\delta_{\rm dp}(m), m} - P_my^\delta\|}\omd \ge \frac{1}{2\tau e^4}\omd
	\end{equation}
	
	for all $m\ge \max\left(\frac{2}{\tau-1},\left(\frac{21}{4}\right)^2\right)m_*$. In order to prove this claim we need another auxiliary definition. For $k\in\N$ we set
	
	\begin{equation}
	m_k:= \lfloor x_{\frac{1}{k-1}}\rfloor= \max\left\{m\in\N~:~\sigma_m^2>(k-1)^{-1}\right\}.
	\end{equation}

	We make the following important assertion:
	
	\begin{equation}\label{lem1:eq3}
	m_{k_{dp}^\delta(m)}\omd \le m_*\omd
	\end{equation}
	for all $m\ge \left(\frac{2}{\tau-1}\right)^2m_*$. To prove assertion \eqref{lem1:eq3}  we assume w.l.o.g. that $k=k_{dp}^\delta(m)>1$. Then the defining relation of the discrepancy principle together with the above assumption that $\lambda\mapsto \sqrt{\lambda}\varphi(\lambda)$ is a qualification of Landweber iteration once more yields
	
	\begin{align*}
	\frac{\tau-1}{2}\sqrt{m}\delta\omd&< \sqrt{\sum_{j=1}^m(1-\sigma_j^2)^{2(k-1)}(y^\dagger,u_j)^2}\le\sqrt{\sum_{j=1}^m(1-\sigma_j^2)^{2(k-1)}\sigma_j^2\varphi^2(\sigma_j^2)(\xi,v_j)^2}\\
	 &\le \sqrt{\sum_{0<\lambda\le \|K\|^2}(1-\lambda)^{2(k-1)}\lambda \varphi^2\left(\lambda\right) \sum_{j=1}^m(\xi,v_j)^2}\\
	 &\le  \sqrt{\frac{1}{k-1}\varphi^2\left(\frac{1}{k-1}\right)\rho^2}\le \sqrt{\sigma^2_{m_k}\varphi^2\left(\sigma^2_{m_k}\right)\rho^2}.
	\end{align*}
	
	Therefore,
	
	\begin{align*}
	\sigma_{m_{k}}^2 \varphi^2\left(\sigma_{m_{k}}^2\right) \rho^2\omd &\ge \left(\frac{\tau-1}{2}\right)^2 m\delta^2\omd \ge m_*\delta^2\omd\\
	&\ge \frac{1}{\Theta(\alpha_*)} \alpha_*\varphi^2(\alpha_*) \delta^2\omd = \alpha_*\varphi^2\left(\alpha_*\right) \rho^2\omd
	\end{align*}
	
	by definition and monotonicity of $\alpha\mapsto \alpha\varphi^2(\alpha)$. Again by monotonicity, the above inequality implies that $\alpha_*\omd\le \sigma_{m_k}^2$ and therefore $m_*\omd \ge x_{\alpha_*}\omd \ge m_k\omd$, which itself implies \eqref{lem1:eq3}.
	
		We move on to the proof of the claim \eqref{lem1:eq2}. Inserting the definition of the discrepancy principle for the nominator  and using  \eqref{lem1:eq3} we have that

	\begin{align*}
	&\frac{\|P_mKx^\delta_{2k^\delta_{\rm dp}(m), m} - P_my^\delta\|}{\|Kx^\delta_{k^\delta_{\rm dp}(m), m} - P_my^\delta\|}\omd\\
	 \ge &\frac{\sqrt{\sum_{j=1}^m(1-\sigma_j^2)^{4k}(y^\delta,u_j)^2}}{\tau \sqrt{m}\delta}\omd\ge \frac{1}{\tau\sqrt{m}\delta}\sqrt{\sum_{j=m_*+1}^m(1-\sigma_j^2)^{4k}(y^\delta,u_j)^2}\omd\\
	&\ge \frac{(1-\sigma_{m_*+1}^2)^{-2k}}{\tau\sqrt{m}\delta}\sqrt{\sum_{j=m_*+1}^m(y^\delta,u_j)^2}\omd\\
	&\ge \frac{1}{(1-(k-1)^{-1})^{2k}\tau\sqrt{m}\delta} \left(\sqrt{\sum_{j=m_*+1}^m(y^\delta-y^\dagger,u_j)^2}-\sqrt{\sum_{j=m_*+1}^m(y^\dagger,u_j)^2}\right)\omd\\
	&\ge \frac{1}{e^4\tau\sqrt{m}\delta} \left(\sqrt{\sum_{j=1}^m(y^\delta-y^\dagger,u_j)^2}-\sqrt{\sum_{j=1}^{m_*}(y^\delta-y^\dagger,u_j)^2}\right.\\
	&\qquad\qquad\qquad\qquad\qquad\qquad \left.-\sqrt{\sum_{j=m_*+1}^m(y^\dagger,u_j)^2}\right)\omd\\
		&\ge \frac{1}{e^4\tau\sqrt{m}\delta}\left(\frac{9}{10}\sqrt{m}\delta\omd - \frac{11}{10}\sqrt{m_*}\delta- \sqrt{\sum_{j=m_*+1}^\infty\sigma_j^2\varphi^2(\sigma_j^2)(\xi,v_j)^2}\right)\omd\\
	&\ge \frac{1}{e^4\tau\sqrt{m}\delta}\left(\frac{9}{10}\sqrt{m}\delta\omd - \frac{11}{10}\sqrt{m_*}\delta- \sigma_{m_*+1}\varphi(\sigma_{m_*+1}^2) \rho\right)\omd\\
    &\ge \frac{1}{e^4\tau\sqrt{m}\delta}\left(\frac{9}{10}\sqrt{m}\delta\omd - \frac{11}{10}\sqrt{m_*}\delta- \sqrt{m_*} \delta\right)\omd\\
    \ge&\frac{1}{e^4\tau\sqrt{m}\delta}\left(\frac{9}{10}-\frac{21}{10}\frac{4}{21}\right)\sqrt{m}\delta\omd
     \ge\frac{1}{2e^4\tau}\omd
	\end{align*}
	
	and the claim is proven. Now we are in position to finish the proof of the Lemma. First \eqref{lem1:eq1} implies that $m^n\omd\le A_\tau \max(m_*,m^{n-1})\omd$ for all $n\in\N$. Moreover, \eqref{lem1:eq3} implies that as soon as $m^n\omd\ge \max(\frac{1}{\tau-1},\left(\frac{21}{4}\right)^2)m_*\omd$ there holds
	
	$$\frac{\|Kx^\delta_{2k^\delta_{\rm dp}(m^n), m^n} - P_my^\delta\|}{\|Kx^\delta_{k^\delta_{\rm dp}(m^n), m^n} - P_{m^n}y^\delta\|}\omd \ge \frac{1}{2\tau e^4}\omd.$$

	Therefore, $n_*$ is finite and we have $m^{n_*}\omd\le A_\tau\max\left(\frac{2}{\tau-1},\left(\frac{21}{4}\right)^2\right)m_*$. The proof of the Lemma is concluded.
\end{proof}

We come to the main proof of Theorem \ref{th2} and for that split the error into three parts, a data propagation error, an approximation error and a discretisation error. Hereby, this can be seen as a bias-variance decomposition where the first term belongs to the variance part and the other two to the bias part. So, 

\begin{align*}
&\|\xdp -x^\dagger\|\\ \le &\sqrt{\sum_{j=1}^\mdp \frac{(1-(1-\sigma_j^2)^\kdp)^2}{\sigma_j^2}(y^\delta-y^\dagger,u_j)^2} +\sqrt{ \sum_{j=1}^\mdp(1-\sigma_j^2)^{2\kdp}(x^\dagger,v_j)^2}\\
&\qquad\qquad\qquad  + \sqrt{\sum_{j=\mdp+1}^\infty(x^\dagger,v_j)^2}.
\end{align*}

We treat the three terms individually. We start with the data propagation error. With the estimate \eqref{pr:eq00} we get 

\begin{align*}
&\sqrt{\sum_{j=1}^{m_{dp}^\delta}(1-(1-\sigma_j^2)^k)^2\sigma_j^{-2}(y^\delta-y^\dagger,u_j)^2}\omd\\
 \le &\sup_{0<\lambda\le \|K\|^2}|1-(1-\lambda)^\kdp|\sqrt{\lambda^{-1}}\sqrt{\sum_{j=1}^\mdp(y^\delta-y^\dagger,u_j)^2}\omd\le \frac{\tau+1}{2} \sqrt{m_{dp}^\delta k_{dp}^\delta} \delta\omd
\end{align*}

and distinguish the cases  $k_{dp}^\delta-1 \ge \alpha_*^{-1}$ and $k_{dp}^\delta-1<\alpha_*^{-1}$. For the first case we once more use the defining relation of the discrepancy principle with the fact that $\lambda\mapsto \sqrt{\lambda}\varphi(\lambda)$ is a qualification of Landweber iteration to obtain

\begin{align*}
\sqrt{m_{dp}^\delta} \delta\omd &< \frac{2}{\tau-1} \sqrt{\sum_{j=1}^{m_{dp}^\delta}(1-\sigma_j^2)^{2(k_{dp}^\delta-1)}(y^\dagger,u_j)^2} \le\frac{2}{\tau-1}\sqrt{ \frac{1}{k_{dp}^\delta-1}\varphi^2\left(\frac{1}{k_{dp}^\delta-1}\right) \rho^2},
\end{align*}

which in turn implies

\begin{align*}
\sqrt{m_{dp}^\delta k_{dp}^\delta} \delta \omd&\le 2\sqrt{\mdp (\kdp-1)} \delta\omd<\frac{4}{\tau-1}\varphi\left(\frac{1}{\kdp-1}\right)\rho\le \frac{4}{\tau-1} \varphi\left(\alpha_*\right) \rho.
\end{align*}

In the latter case we use that $m_{dp}^\delta\omd\le B_\tau m_*$ by Lemma \ref{lem1}, thus

$$\sqrt{m_{dp}^\delta k_{dp}^\delta} \delta \le \sqrt{B_\tau} \sqrt{\frac{m_*}{\alpha_*}}\delta \le \sqrt{B_\tau} \sqrt{\frac{x_{\alpha_*}}{\alpha_*}}\delta=\sqrt{B_\tau}\varphi\left(\alpha_*\right)\rho$$

by definition of $m_*$. Therefore the data propagation error is bounded by 

$$\sqrt{\sum_{j=1}^{m_{dp}^\delta}(1-(1-\sigma_j^2)^k)^2\sigma_j^{-2}(y^\delta-y^\dagger,u_j)^2}\omd\le \sqrt{B_\tau}\varphi\left(\alpha_*\right) \rho = \sqrt{B_\tau}\varphi\left(\Theta^{-1}\left(\frac{\delta^2}{\rho^2}\right)\right)\rho.$$

We move on to the approximation error, for which we employ a conditional stability estimate. By Proposition 2 of \cite{hohage2000regularization} it holds that

\begin{align}\notag
&\sum_{j=1}^{m_{dp}^\delta}(1-\sigma_j^2)^{2k_{dp}^\delta}(x^\dagger,v_j)^2 \omd\\
\le &\rho^2 \phi^{-1}\left(\frac{\sum_{j=1}^{m_{dp}^\delta}\sigma_j^2(1-\sigma_j^2)^{2\kdp}(x^\dagger,v_j)^2}{\rho^2}\right)\omd
\le \rho^2 \phi^{-1}\left(\frac{\sum_{j=1}^\mdp(y^\dagger,u_j)^2}{\rho^2}\right)\\
\le &\rho^2\phi^{-1}\left(\frac{2}{\rho^2}\left(\sum_{j=1}^\mdp(y^\delta,u_j)^2 + \sum_{j=1}^\mdp(y^\delta-y^\dagger,u_j)^2\right)\right)\\\notag
\le &\rho^2\phi^{-1}\left(\frac{2\left(\tau^2+\frac{(\tau+1)^2}{4}\right)m_{dp}^\delta \delta^2}{\rho^2}\right)\omd\le \rho^2 4\tau^2B_\tau\phi^{-1}\left(\frac{m_* \delta^2}{\rho^2}\right)\\\label{th2:eq2}
 = &4\tau^2 B_\tau \rho^2 \phi^{-1}\left(\alpha_* \varphi^2(\alpha_*)\right)= 4\tau^2 B_\tau \rho^2 \varphi^2(\alpha_*) =4\tau^2B_\tau \varphi^2\left(\Theta^{-1}\left(\frac{\delta^2}{\rho^2}\right)\right)\rho^2,
\end{align}

where we used Lemma \ref{lem1}, as well as monotonicity of $\phi^{-1}$ and concavity of $\phi^{-1}$ together with $\phi^{-1}(0)=0$ and $4\tau^2B_\tau\ge1$.
 
Finally we treat the discretisation error.  We first make the following observation. It holds that

\begin{equation}\label{th2:eq1}
\kdp\omd \le \frac{4+\log(2\tau)}{\sigma_\mdp^2}.
\end{equation}

We prove the assertion \eqref{th2:eq1} by contradiction and  make the assumption that $\kdp> (4+\log(2\tau))\sigma_\mdp^{-2}\omd$. Then, by monotonicity of $(\sigma_j^2)_j$,

\begin{align*}
\frac{\|Kx^\delta_{2\kdp, \mdp} - P_{\mdp}y^\delta\|}{\|Kx^\delta_{\kdp, \mdp} - P_{\mdp}y^\delta\|}\omd&= \sqrt{\frac{\sum_{j=1}^\mdp(1-\sigma_j^2)^{4\kdp}(y^\delta,u_j)^2}{\sum_{j=1}^\mdp(1-\sigma_j^2)^{2\kdp}(y^\delta,u_j)^2}}\omd
\\\notag
\le &(1-\sigma_{\mdp}^2)^{\kdp}\omd\le e^{-\kdp\sigma_\mdp^2}\omd
< e^{-4-\log(2\tau)}=\frac{1}{2\tau e^4},
\end{align*}

which contradicts the definition of $n_*$ and therefore proves \eqref{th2:eq1}. We have that by definition of $m_*$ there holds

\begin{align*}
\sum_{j=m_{dp}^\delta+1}^\infty(x^\dagger,v_j)^2 &= \sum_{j=m_{dp}^\delta+1}^{m_*}(x^\dagger,v_j)^2 + \sum_{j=m_*+1}^\infty\varphi^2(\sigma_j^2)(\xi,v_j)^2\\
 &\le \sum_{j=m_{dp}^\delta+1}^{m_*}(x^\dagger,v_j)^2 + \varphi^2(\sigma_{m_*+1}^2)\rho^2\\
 &\le \sum_{j=\mdp+1}^{m_*}(x^\dagger,v_j)^2 + \varphi^2\left(\Theta^{-1}\left(\frac{\delta^2}{\rho^2}\right)\right)\rho^2
\end{align*}

where we use the convention that $\sum_{j=n}^m = 0$ for $n>m$. We observe that only the first summand in the case that $\mdp<m_*$ has to be investigated. Using \eqref{th2:eq1} in the third step below yields

\begin{align*}
\sum_{j=m_{dp}^\delta+1}^{m_*}(x^\dagger,v_j)^2\omd &\le (1-\sigma_{m_{dp}^\delta}^2)^{-2k_{dp}^\delta} \sum_{j=m_{dp}^\delta+1}^{m_*}(1-\sigma_j^2)^{2k_{dp}^\delta}(x^\dagger,v_j)^2\omd\\
&\le 2\left(e^{k_{dp}^\delta \sigma_{m_{dp}^\delta}^2}\right)^2 \sum_{j=m_{dp}^\delta+1}^{m_*}(1-\sigma_j^2)^{2k_{dp}^\delta}(x^\dagger,v_j)^2\omd\\
&\le 8e^8\tau^2 \sum_{j=1}^{m_*}(1-\sigma_j^2)^{2k_{dp}^\delta}(x^\dagger,v_j)^2.
\end{align*}

By construction of $m_{dp}^\delta$, we have that for $m_{dp}^\delta< m_*$ there holds $k_{dp}^\delta \ge k_{dp}^\delta(m_*)$. Finally, by the same arguments used in \eqref{th2:eq2} above

\begin{align*}
\sum_{j=1}^{m_*}(1-\sigma_j^2)^{2k_{dp}^\delta}(x^\dagger,v_j)^2 \omd&\le \sum_{j=1}^{m_*}(1-\sigma_j^2)^{2k_{dp}^\delta(m_*)}(x^\dagger,v_j)^2 \omd\\
&\le 4\tau^2\rho^2 \phi^{-1}\left(\alpha_*\varphi^2(\alpha_*)\right)
= 4 \tau^2 \varphi^2\left(\Theta^{-1}\left(\frac{\delta^2}{\rho^2}\right)\right)\rho^2
\end{align*}

and putting all three estimates together shows that there exists a constant $L$ with

$$\|x_{k_{dp}^\delta,m_{dp}^\delta}^\delta - x^\dagger\| \omd\le L\rho \varphi\left(\Theta^{-1}\left(\frac{\delta^2}{\rho^2}\right)\right).$$

This finishes the proof of Theorem \ref{th2}.

\subsection{Proof of Theorem \ref{th1}}

The ultimate goal will be to show that 

\begin{equation}\label{th0:eq0}
\mathbb{P}\left( \mdp= 1\right)\ge \frac{1}{48}
\end{equation}

for all $\delta>0$ small enough, since if $\mdp= 1$ by definition of $x^\dagger$ there clearly holds

$$\|\xdp - x^\dagger\|\ge \sqrt{\sum_{j=2}^\infty(x^\dagger,u_j)^2}=\sqrt{\|v_2\|^2/2}=1/\sqrt{2}$$

for all $\delta>0$, because $\xdp\in{\rm span}(v_1,...,v_{\mdp})$. We define the event 

\begin{align}\label{th0:eq1}
\Omega_{\delta}:&=\left\{0<(Z,u_1)\le\left(\sqrt{\frac{9}{8}}-1\right) \frac{\sigma_1}{\delta}\right\}\\
&\qquad\qquad\cap \left\{-\sigma_1\le(Z,u_2)\le 0\right\}\cap \left\{\sum_{j=3}^m(Z,u_j)^2 \le \tau^2 (m-2),~\forall m\ge 3\right\}.
\end{align}

By independence, 

\begin{align*}
&\mathbb{P}\left(\Omega_\delta\right)\\= &\mathbb{P}\left(0<(Z,u_1)\le\left(\sqrt{\frac{9}{8}}-1\right) \frac{\sigma_1}{\delta}\right) \mathbb{P}\left(-\sigma_1\le(Z,u_2)\le 0\right)\\
&\qquad\qquad\qquad\cdot\mathbb{P}\left(\sum_{j=3}^m(Z,u_j)^2 \le \tau^2 (m-2),~\forall m\ge 3\right)\\
=&\left(\Phi\left(\left(\sqrt{\frac{9}{8}}-1\right) \frac{\sigma_1}{\delta}\right)-\Phi(0)\right)\left(\Phi(0)-\Phi\left(\frac{\sigma_1}{\delta}\right)\right)\left(1-\frac{1}{\tau^2-1}\E\left|(Z,u_1)^2-1\right|\right)\\
\ge &\left(\Phi\left(\left(\sqrt{\frac{9}{8}}-1\right) \frac{\sigma_1}{\delta}\right)-\Phi(0)\right)\left(\Phi(0)-\Phi\left(\frac{\sigma_1}{\delta}\right)\right)\left(1-\frac{2}{4-1}\right)\\
\ge &\frac{1}{4} \frac{1}{4}\frac{1}{3}=\frac{1}{48}
\end{align*}

for $\delta$ small enough, where $\Phi$ is the cumulative distribution function of a standard Gaussian and where we have used Proposition \ref{prop0} in the second step (note that $\delta(y^\delta-y^\dagger,u_j) = (Z,u_j)$ and shift the index). There exist $\delta>0$ arbitrarily small such that it is possible to choose $k_\delta\in\N$ with

$$\tau^2\delta^2\le(1-\sigma_1^2)^{2k_\delta}\sigma_1^2\le \frac{4}{3}\tau^2\delta^2.$$

We observe that \eqref{th0:eq0} follows, if we can show that

\begin{align}\label{th0:eq4}
\sum_{j=1}^1(1-\sigma_j^2)^{2k_\delta}(y^\delta,u_j)^2\chi_{\Omega_\delta}&> \tau^2 \delta^2\chi_{\Omega_\delta},\\\label{th0:eq5}
\sum_{j=1}^m(1-\sigma_j^2)^{2k_\delta}(y^\delta,u_j)^2\chi_{\Omega_\delta}&\le \tau^2m\delta^2\chi_{\Omega_\delta},~\mbox{for all }m\ge 2
\end{align}

for the above  delta's, because then $\kdp(1)\chi_{\Omega_\delta}>k_\delta\chi_{\Omega_\delta}$ and $\kdp(m)\chi_{\Omega_\delta}\le k_\delta\chi_{\Omega_\delta}$ for all $m\ge 2$ and therefore $\mdp\chi_{\Omega_\delta} = \arg\max_{m\in\N} \kdp(m)\chi_{\Omega_\delta}= 1$. First, since  

$$(1-\sigma_1^2)^{2k_\delta}(y^\delta,u_1)^2=(1-\sigma_1^2)^{2k_\delta}\left((y^\dagger,u_1) + \delta(Z,u_1)\right)^2$$
there holds

\begin{align*}
(1-\sigma_1^2)^{2k_\delta}(y^\delta,u_1)^2\chi_{\Omega_\delta}&\le(1-\sigma_1^2)^{2k_\delta}\left(\sigma_1 + \delta\left(\sqrt{\frac{9}{8}}-1\right)\frac{\sigma_1}{\delta}\right)^2 =\frac{4}{3}\tau^2\delta^2 \frac{9}{8}\le \frac{3}{2}\tau^2\delta^2,\\
(1-\sigma_1^2)^{2k_\delta}(y^\delta,u_1)^2\chi_{\Omega_\delta}&>(1-\sigma_1^2)^{2k_\delta}(\sigma_1+0)^2= \tau^2\delta^2\chi_{\Omega_\delta}
\end{align*}

by definition of $k_\delta$ and $\Omega_\delta$; note that $(y^\dagger,u_j)=\sigma_j(x^\dagger,v_j)$. The assertion \eqref{th0:eq4} follows from the second line. Similar, since $\sigma_2=\sigma_1$ there holds

\begin{equation*}
(1-\sigma_2^2)^{2k_\delta}(y^\dagger,u_2)\chi_{\Omega_\delta}\le (1-\sigma_1)^{2k_\delta}\sigma_1^2\frac{1}{2} = \frac{2}{3}\tau^2\delta^2.
\end{equation*}

We deduce 

\begin{align*}
&\sum_{j=1}^m(1-\sigma_j^2)^{2k_\delta}(y^\delta,u_j)^2\chi_{\Omega_\delta}\\ = &(1-\sigma_1^2)^{2k_\delta}\sigma_1^2(y^\delta,u_1)^2\chi_{\Omega_\delta} + (1-\sigma_2^2)^{2k_\delta}\sigma_2^2(y^\delta,u_2)^2\chi_{\Omega_\delta}+\sum_{j=3}^m(1-\sigma_j^2)^{2k_\delta}(y^\delta,u_j)^2\chi_{\Omega_\delta}\\
\le& \frac{4}{3}\tau^2\delta^2+\frac{2}{3}\tau^2\delta^2 + \tau^2(m-2)\delta^2= \tau^2m\delta^2,
\end{align*}

which proves \eqref{th0:eq5}. The proof is accomplished.

\subsection{Proof of the corollaries}
In this section we prove the four corollaries. Corollary \ref{cor1} directly follows from Theorem \ref{th2} together with a well-known result for general source conditions. Corollary 2 from \cite{mathe2008general} states that for every $x^\dagger\in\mathcal{X}$ there exists a concave index function $\varphi$ and an element $\xi\in\mathcal{X}$ with $x^\dagger= \varphi(K^*K)\xi$; note that the result there is formulated for injective operators, which is no restriction, since $K^+y^\dagger\in\mathcal{N}(K)^\perp$ and hence we can replace $K$ with the restriction of $K$ onto $\mathcal{N}(K)^\perp$. Clearly we can assume that $\varphi$ is strictly monotonically increasing. In order to finish the proof of Corollary \ref{cor1} we have to show that $\varphi$ fulfills Assumption \ref{as1}. Let $C_k$ be such that 

\begin{align*}
\sup_{0<\lambda< 1}(1-\lambda)^k\varphi(\lambda) = \left(1-\frac{C_k}{k}\right)^k\varphi\left(\frac{C_k}{k}\right)
\end{align*}

If $C_k\le 1$ there holds 

\begin{align}\label{cor1:eq1}
\left(1-\frac{C_k}{k}\right)^k\varphi\left(\frac{C_k}{k}\right)\le \varphi\left(\frac{1}{k}\right).
\end{align}

Otherwise, since $\varphi$ is concave we have that

\begin{align}\label{cor1:eq2}
\left(1-\frac{C_k}{k}\right)^k\varphi\left(\frac{C_k}{k}\right)\le\left(1-\frac{C_k}{k}\right)^kC_k\varphi\left(\frac{1}{k}\right)\le e^{-C_k}C_k\varphi\left(\frac{1}{k}\right)\le \varphi\left(\frac{1}{k}\right)
\end{align}

and \eqref{cor1:eq1} and \eqref{cor1:eq2} together imply that $\varphi$ is a qualification of Landweber iteration.  Since $\varphi$  is concave and strictly monotonically increasing it follows that $\varphi^{-1}$ is convex and strictly monotonically increasing and therefore also $f(x):=\left(\varphi^{-1}\right)^2(x)$. Then for $g(x):=xf(x)$ and $\lambda\in[0,1]$ there holds

\begin{align*}
&g(\lambda x+ (1-\lambda)y)\\
 \le &(\lambda x + (1-\lambda) y )(\lambda f(x)+(1-\lambda) f(y))\\
=&\lambda x f(x)+(1-\lambda)y f(y) - \lambda(1-\lambda)x  f(x) -\lambda(1-\lambda)yf(y)\\
&\qquad\qquad\qquad\qquad+\lambda(1-\lambda)yf(x)+\lambda(1-\lambda)xf(y)\\
&=\lambda x f(x)+(1-\lambda)y f(y) + \lambda(1-\lambda)(y-x)\left(f(x)-f(y)\right)\\
&\le \lambda x f(x)+(1-\lambda)y f(y) = \lambda g(x)+(1-\lambda)g(y),
\end{align*}

where we used convexity of $f$ and the fact that it is increasing in the first and fourth step. Thus $g$ is convex and consequently Assumption \ref{as1} is fulfilled. Finally, Theorem \ref{th2} implies
$$\mathbb{P}\left(\|\xdp - x^\dagger\|\le \|\xi\|^2 \varphi^2\left(\Theta^{-1}\left(\frac{\delta^2}{\|\xi\|^2}\right)\right)\right)\to 1,$$
 as $\delta/\rho\to0$. The proof of Corollary \ref{cor1} is finished with the fact that $\varphi^2\left(\Theta^{-1}(x)\right)\to 0$ for $x\to0$.
 
  Corollary \ref{cor2} can be deduced from Theorem \ref{th1} and the following proposition.

\begin{proposition}\label{prop}
	For either polynomially ill-posed problems under H\"older source conditions or exponentially ill-posed problems under logarithmic source conditions there exists $c>0$ such that
	
	$$\min_{k,m\in\N} \sup_{x^\dagger\in\mathcal{X}_{\varphi,\rho}}\|x_{k,m}^\delta-x^\dagger\|\omd \ge c \rho \varphi\left(\Theta^{-1}\left(\frac{\delta}{\rho}\right)\right)\omd,$$
	
	with $c=\min\left(\frac{c'}{\sqrt{2}e},\sqrt{\frac{1}{5}}\left(1-e^{-\frac{1}{2}}\right)\right)$ and $c'=c'(q,\nu,p)$ given below.
	
\end{proposition}

\begin{proof}
	We have
	
	\begin{equation}
	\|x_{k,m}^\delta-x^\dagger\|^2=\sum_{j=1}^m\left\{\frac{1-(1-\sigma_j^2)^k}{\sigma_j}(y^\delta-y^\dagger,u_j) + (1-\sigma_j^2)^k(x^\dagger,v_j)\right\}^2 + \sum_{j=m+1}^\infty(x^\dagger,v_j)^2.
	\end{equation}
	
	We argue by contradiction and assume that 
	
	\begin{equation}\label{prop:eq1}
	\min_{k,m\in\N} \sup_{x^\dagger\in\mathcal{X}_{\varphi,\rho}}\|x_{k,m}^\delta-x^\dagger\|\omd < c \rho \varphi\left(\Theta^{-1}\left(\frac{\delta}{\rho}\right)\right)\omd
	\end{equation}
	
	for $\delta/\rho$ small enough. Let $m_*$ be the index defined at the beginning of the section. Remember that either $K$ is polynomially ill-posed under H\"older source condition (i.e., $\sigma_j^2=j^{-q}$ and $\varphi(t)=t^{\nu/2}$) or $K$ is exponentially ill-posed under logarithmic source condition (i.e., $\sigma_j^2=\exp(-aj)$ and $\varphi(t)=\left(-\log(t)\right)^{-p/2}$). Then it holds that
	
	\begin{equation}\label{prop:eq2}
      \varphi^2\left(\sigma_{2m_*+1}^2\right) = \frac{\varphi^2\left(\sigma_{2m_*+1}^2\right)}{\varphi^2\left(\sigma_{m_*+1}^2\right)} \varphi^2\left(\sigma_{m_*+1}^2\right)\ge c' \varphi^2\left(\Theta^{-1}\left(\frac{\delta^2}{\rho^2}\right)\right)
\end{equation}
	
	with $c'=\min\left(2^{-q\nu},2^{-p}\right)$. Consequently,
	
		$$\sup_{x^\dagger\in\mathcal{X}_{\varphi,\rho}}\sum_{j=2m_*+1}^\infty(x^\dagger,v_j)^2 \ge \varphi^2\left(\sigma_{2m_*+1}^2\right)\rho^2 \ge c'\rho^2\varphi^2\left(\Theta^{-1}\left(\frac{\delta^2}{\rho^2}\right)\right),$$

	and we deduce from assumption \eqref{prop:eq1} that an optimal $m$ has to be larger than $2m_*$. Further, 
	
	\begin{align*}
	&\min_{k,m\in\N} \sup_{x^\dagger\in\mathcal{X}_{\varphi,\rho}}\|x_{k,m}^\delta-x^\dagger\|^2\\
	\ge &\min_{k\in\N,m\ge 2m_*}\sup_{x^\dagger\in\mathcal{X}_{\varphi,\rho}}\sum_{j=1}^m\left\{\frac{1-(1-\sigma_j^2)^k}{\sigma_j}(y^\delta-y^\dagger,u_j) + (1-\sigma_j^2)^k(x^\dagger,v_j)\right\}^2\\
		\ge &\min_{k\in\N}\sup_{\xi\in\mathcal{X}_{\varphi,\rho}}\sum_{j=1}^{2m_*}\left\{\frac{1-(1-\sigma_j^2)^k}{\sigma_j}(y^\delta-y^\dagger,u_j) + (1-\sigma_j^2)^k(x^\dagger,v_j)\right\}^2,
	\end{align*}
	
	where we dropped the discretisation error in the first step and used monotonicity in $m$ in the second step. Because of the supremum we can assume that both terms of each summand have the same sign; this will make the sum only larger. Therefore, we expand the square and drop the mixing terms
	
	\begin{align*}
	&\min_{k\in\N}\sup_{x^\dagger\in\mathcal{X}_{\varphi,\rho}}\sum_{j=1}^{2m_*}\left\{\frac{1-(1-\sigma_j^2)^k}{\sigma_j}(y^\delta-y^\dagger,u_j) + (1-\sigma_j^2)^k(x^\dagger,v_j)\right.\\
	= &\min_{k\in\N}\sup_{x^\dagger\in\mathcal{X}_{\varphi,\rho}}\left\{\sum_{j=1}^{2m_*}\frac{(1-(1-\sigma_j^2)^k)^2}{\sigma_j^2}(y^\delta-y^\dagger,u_j)^2+(1-\sigma_j^2)^{2k}(x^\dagger,v_j)^2\right.\\
	&\qquad\qquad\left.  +2(1-\sigma_j^2)^k\frac{1-(1-\sigma_j^2)^k}{\sigma_j}(y^\delta-y^\dagger,u_j)(x^\dagger,v_j)\right\}\\
	&\ge\min_{k\in\N}\sup_{x^\dagger\in\mathcal{X}_{\varphi,\rho}}\left\{\sum_{j=1}^{2m_*}\frac{(1-(1-\sigma_j^2)^k)^2}{\sigma_j^2}(y^\delta-y^\dagger,u_j)^2+\sum_{j=1}^{2m_*}(1-\sigma_j^2)^{2k}(x^\dagger,v_j)^2\right\}
	\end{align*}
	
Finally, let $k_o:=\max\left\{k\in\N~:~ \sigma_{2m_*}^2 \le k^{-1}\right\}$. Then we directly see that for $\delta/\rho$ small enough

\begin{align*}
&\min_{k\le k_o}\sup_{x^\dagger\in\mathcal{X}_{\varphi,\rho}} \sum_{j=1}^{2m_*}(1-\sigma_j^2)^{2k}\varphi^2(\sigma_j^2)(\xi,v_j)^2\\
= &\min_{k\le k_o}\max_{j=1,...,2m_*}(1-\sigma_j^2)^{2k}\varphi^2(\sigma_j^2)\rho^2\ge\rho^2 \min_{k\le k_o}(1-\sigma_{2m_*}^2)^{2k}\varphi^2(\sigma_{2m_*}^2)\\
=&\rho^2(1-\sigma_{2m_*}^2)^{2k_o}\varphi^2(\sigma_{2m_*}^2)\ge \rho^2\frac{e^{-2\sigma_{2m_*}^2k_o}}{2}\varphi^2(\sigma_{2m_*}^2)    \ge \frac{\rho^2}{2e^2} \varphi^2\left(\sigma_{2m_*}^2\right)\\
 \ge &\frac{c'}{2e^2} \rho^2 \varphi^2\left(\Theta^{-1}\left(\frac{\delta^2}{\rho^2}\right)\right)
\end{align*}

and deduce from assumption \eqref{prop:eq2} that an optimal $k$ has to be larger then $k_o$. Now we drop the approximation error and again use monotonicity to obtain that

\begin{align*}
&\min_{k\in\N}\sup_{\xi\in\mathcal{X}_{\varphi,\rho}}\sum_{j=1}^{2m_*}\frac{(1-(1-\sigma_j^2)^k)^2}{\sigma_j^2}(y^\delta-y^\dagger,u_j)^2+ \sum_{j=1}^{2m_*}(1-\sigma_j^2)^{2k}(x^\dagger,v_j)^2 \\
  \ge &\min_{k\ge k_o}\sum_{j=1}^{2m_*}\frac{(1-(1-\sigma_j^2)^k)^2}{\sigma_j^2}(y^\delta-y^\dagger,u_j)^2
\ge \sum_{j=m_*+1}^{2m_* } \frac{(1-(1-\sigma_j^2)^{k_o})^2}{\sigma_j^2}(y^\delta-y^\dagger,u_j)^2\\
\ge &\frac{(1-(1-\sigma_{2m_*}^2)^{k_o})^2}{\sigma_{m_*+1}^2}\sum_{j=m_*+1}^{2m_*}(y^\delta-y^\dagger,u_j)^2\ge \frac{\left(1-e^{-k_o\sigma_{2m_*}^2}\right)^2}{\sigma_{m_*+1}^2}\sum_{j=m_*+1}^{2m_*}(y^\delta-y^\dagger,u_j)^2\\
  = &\frac{\left(1-e^{-(k_o+1)\sigma_{2m_*}^2 \frac{k_o}{k_o+1}}\right)^2}{\sigma_{m_*+1}^2}\sum_{j=m_*+1}^{2m_*}(y^\delta-y^\dagger,u_j)^2 \ge \frac{\left(1-e^{-\frac{k_o}{k_o+1}}\right)^2}{\sigma_{m_*+1}^2} \sum_{j=m_*+1}^{2m_*}(y^\delta-y^\dagger,u_j)^2\\
  \ge &\frac{\left(1-e^{-\frac{1}{2}}\right)^2}{\sigma_{m_*+1}^2}\sum_{j=m_*+1}^{2m_*}(y^\delta-y^\dagger,u_j)^2.
\end{align*}

Finally, by definition of $\Omega_{\delta/\rho}$, 

\begin{align*}
\frac{\sum_{j=m_*+1}^{2m_*}(y^\delta-y^\dagger,u_j)^2}{\sigma_{m_*+1}^2}\omd &= \frac{\sum_{j=1}^{2m_*}(y^\delta-y^\dagger,u_j)^2 - \sum_{j=1}^{m_*}(y^\delta-y^\dagger,u_j)^2}{\sigma_{m_*+1}^2}\omd\\
 &\ge \frac{1}{\sigma_{m_*+1}^2}\left(\left(\frac{9}{10}\right)^22m_*\delta - \left(\frac{11}{10}\right)^2m_* \delta^2\right)\omd\\
 &\ge \frac{2}{5}\frac{m_*}{m_*+1}\frac{(m_*+1)\delta^2}{\sigma_{m_*+1}^2}\omd\ge \frac{1}{5}\varphi^2\left(\Theta^{-1}\left(\frac{\delta^2}{\rho^2}\right)\right)\rho^2\omd
\end{align*}

Putting the above two estimates together yields

\begin{align*}
	&\min_{k,m\in\N} \sup_{x^\dagger\in\mathcal{X}_{\varphi,\rho}}\|x_{k,m}^\delta-x^\dagger\|^2\omd	\ge \frac{\left(1-e^{-\frac{1}{2}}\right)^2}{5} \varphi^2\left(\Theta^{-1}\left(\frac{\delta^2}{\rho^2}\right)\right)\rho^2\omd,
\end{align*}

which contradicts the assumption \eqref{prop:eq2} and concludes the proof.
\end{proof}

\begin{proof}[Proof of Corollary \ref{cor3}]
It suffices to show that there exists $C=C(\varepsilon)$ such that
	
	\begin{align}\label{cor3:eq1}
	\mathbb{P}\left(\frac{\|Kx^\delta_{2k^\delta_{\rm dp}(m), m} - P_my^\delta\|}{\|Kx^\delta_{k^\delta_{\rm dp}(m), m} - P_my^\delta\|}\ge \frac{1-\varepsilon}{\tau}\right)\to 1
	\end{align}
	
	for all $m\ge Cm_*$, as $\delta/\rho\to 0$.	From \eqref{lem1:eq3} and the explicit form of the $\sigma_j$'s one can deduce that there exists $C_1\in\N$ such that for   $m'= C_1m_*$ there holds 
	
	\begin{align*}
	(1-\sigma_{m_*+1}^2)^{2\kdp(m)}\omd&\ge	e^{-\sigma_{m'+1}^2\kdp(m)}\omd\ge e^{-\sigma_{m'+1}^2\sigma_{m_*}^{-2}}\omd\\
	&\ge\omd \begin{cases}e^{-\frac{1}{C_1^q}} &\mbox{, for } \sigma_j^2=j^{-q}\\
		             e^{-e^{a(1-C_1)m_*}} &\mbox{, for } \sigma_j^2 = e^{-aj}\end{cases}\\
		             &\ge   \sqrt{1-\varepsilon}\omd.
	\end{align*}

	Further, there exist $\varepsilon'>0$ and $K_2\in\N$ such that for all $m\ge K_2m'=K_2K_1m_*$ there holds 
	
    \begin{align*}
	(1-\varepsilon')\sqrt{m}-(2+\varepsilon')\sqrt{m'} &= \left(1-\varepsilon'-(2+\varepsilon')\sqrt{\frac{m'}{m}}\right)m\\
	&\ge\left(1-\varepsilon' - \frac{2+\varepsilon'}{\sqrt{K_2}}\right) m   \ge \sqrt{1-\varepsilon}m.
	\end{align*}
	
	We define 
	
	\begin{equation*}
	\Omega_{\delta/\rho}':=\Omega_{\delta/\rho} \cap\left\{\left|\sqrt{\sum_{j=1}^m(y^\delta-y^\dagger,u_j)^2}-\sqrt{m}\delta\right|\le   \varepsilon'\sqrt{m}\delta,~\forall m\ge m_*\right\}.
	\end{equation*}

	Note that $\mathbb{P}\left(\Omega_{\delta/\rho}'\right)\to 0$ as $\delta/\rho\to 0$ by Proposition \ref{prop0}. Similar as in the proof of Lemma \ref{lem1} we deduce
	
		\begin{align*}
	&\frac{\|Kx^\delta_{2k^\delta_{\rm dp}(m), m} - P_my^\delta\|}{\|Kx^\delta_{k^\delta_{\rm dp}(m), m} - P_my^\delta\|}\chi_{\Omega'_{\delta/\rho}}\\
	&\ge \frac{(1-\sigma_{m'+1}^2)^{-2k}}{\tau\sqrt{m}\delta}\sqrt{\sum_{j=m'+1}^m(y^\delta,u_j)^2}\chi_{\Omega'_{\delta/\rho}}\\
	&\ge \frac{\sqrt{1-\varepsilon'}}{\tau\sqrt{m}\delta}\left(\sqrt{\sum_{j=1}^m(y^\delta-y^\dagger,u_j)^2}-\sqrt{\sum_{j=1}^{m_*}(y^\delta-y^\dagger,u_j)^2}-\sqrt{\sum_{j=m_*+1}^m(y^\dagger,u_j)^2}\right)\chi_{\Omega'_{\delta/\rho}}\\
	&\ge \frac{1}{\tau\sqrt{m}\delta}\left((1-\varepsilon')\sqrt{m}\delta - (2+\varepsilon')\sqrt{m_*}\delta\right)\chi_{\Omega'_{\delta/\rho}}\ge \frac{1-\varepsilon'}{\tau}\chi_{\Omega'_{\delta/\rho}}
	\end{align*}
	
	for all $m\ge K_1K_2 m_*$, thus \eqref{cor3:eq1} holds for $C=K_1K_2$.
	
\end{proof}

\begin{proof}[Proof of Corollary \ref{cor4}]
	
	We give only a sketch. We focus on two things, namely that first the \texttt{while}-loop terminates with a discretisation dimension $\mdp\le Cm_*$ for some $C$ large enough, and that 
	
	\begin{equation}
		\alpha^\delta_{\rm dp}\omd\ge \sigma_{m_*}^2\omd.
	\end{equation}

	Let $ \alpha>0$, $t=1/8\tau$  and $m\in\N$ with $\alpha\le \sigma_m^2$. Then there holds

	\begin{align*}
	\max_{j=1,...,m}\frac{\sigma_j^2+\alpha}{\sigma_j^2+t\alpha}\le \frac{\sigma_m^2+\alpha}{\sigma_m^2+t\alpha}< \frac{\sigma_m^2+\sigma_m^2}{\sigma_m^2}\le 2.
	\end{align*}

	Therefore,

	\begin{align}\label{subsec:eq3}
	\frac{\|P_mK x_{t\alpha,m}^\delta-P_my^\delta\|}{\|P_mKx_{\alpha,m}^\delta - P_my^\delta\|} = \sqrt{\frac{\sum_{j=1}^m\frac{t^2\alpha^2}{(t\alpha+\sigma_j^2)^2}(y^\delta,u_j)^2}{\sum_{j=1}^m\frac{\alpha^2}{(\alpha+\sigma_j^2)^2}(y^\delta,u_j)^2}}< t \max_{j=1,...,m}\frac{\sigma_j^2+\alpha}{\sigma_j^2+t\alpha}\le 2t=\frac{1}{4\tau},
	\end{align}
	
	whenever $\alpha\le \sigma_m^2$. Now we show that for $C=\left(\frac{2}{(\tau-1)\sqrt{t}^3}\right)^2$ there holds

	\begin{equation}\label{subsec:eq2}
	\alpha_{\rm dp}^\delta(m)\omd\ge \sigma_{m_*}^2\omd
	\end{equation}
	
	for all $m\ge Cm_*$, and we proceed in a similar fashion as in \eqref{lem1:eq3}. For $\alpha=\alpha^\delta_{\rm dp}(m)$ there holds
	
	\begin{align*}
	\frac{\tau-1}{2}\sqrt{m}\delta\omd &\le \sqrt{\alpha \varphi^2(\alpha) \rho^2} = t^{-\frac{3}{2}}\sqrt{t\alpha \left(t\varphi(\alpha)\right)^2\rho^2}\le t^{-\frac{3}{2}} \sqrt{t\alpha \varphi^2(t\alpha)\rho^2}
	\end{align*}
	
	by concavity (note that $t<1$), thus
	
	\begin{align*}
	t\alpha \varphi^2(t\alpha)\rho^2\omd &\ge \left(\frac{\tau-1}{2}\right)^2 t^3 m \delta^2\omd \ge m_* \delta^2\omd\ge \sigma_{m_*}^2\varphi^2\left(\sigma_{m_*}^2\right)\rho^2
	\end{align*}
	
	and this implies \eqref{subsec:eq2} by monotonicity.  Consequently,

	\begin{align*}
	\frac{\|P_mK x^\delta_{t\alpha,m} - P_my^\delta\|}{\|P_mK x^\delta_{\alpha,m}-P_my^\delta\|}\omd&\ge \frac{1}{\tau\sqrt{m}\delta} \sqrt{\sum_{j=1}^m\frac{t^2\alpha^2}{(t\alpha+\sigma_j^2)^2}(y^\delta,u_j)^2}\omd\\
	&\ge \frac{1}{\tau\sqrt{m}\delta}\frac{t\alpha}{t\alpha+\sigma_{m_*+1}^2} \sqrt{\sum_{j=m_*+1}^m(y^\delta,u_j)^2}\omd\\
	&\ge \frac{1}{\tau\sqrt{m}\delta}\frac{t}{t+t}\left(\frac{9}{10}-\frac{21}{10}\frac{4}{21}\right)\sqrt{m}\omd&\\
	&\ge \frac{1}{4\tau}\omd
	\end{align*}
	
	for all $m\ge Cm_*$. We deduce that the \texttt{while}-loop terminates with $\mdp\le Cm_*$ and moreover, \eqref{subsec:eq3} implies \eqref{subsec:eq2}.
	
\end{proof}

\section{Numerical experiments}\label{sec:5}
In this section we investigate our method numerically. We treat four classic test problems from the popular open-source Matlab toolbox  \cite{hansen1994regularization}, namely \texttt{phillips}, \texttt{deriv2}, \texttt{gravity} and \texttt{heat}. These are discretisations of Fredholm integral equations via Galerkin box functions and quadrature rules and they cover different types of ill-posedness and solution smoothness.  They constitute of a discretisation of the forward integral operator $K$ into a matrix $A\in\R^{D\times D}$ and discretisations of the exact data $y^\dagger\in \R^{D}$ and the exact solution $x^\dagger\in\R^D$. As measured data we set 

$$(y^\delta,e_j)= (y^\dagger,e_j)+\delta Z_j,\qquad j=1,...,D$$

where $(\cdot,e_j)$ is the $j$-th coordinate and $Z_j$ is i.i.d standard Gaussian white noise with noise level $\delta\in\{ 1,10^{-2},10^{-4}\}$.  We compare our method to the aforementioned sequential early stopping discrepancy principle \cite{blanchard2018optimal}, given by

$$ k^\delta_{\rm es}:=\min\left\{k\in \N~:~ \|A x_k^\delta-y^\delta\|\le \sqrt{D}\delta \right\}.$$

In contrast to the theoretical part, here we do not want to impose the restrictive assumption that we have the singular valued decomposition at hand for discretisation. Instead we rely on the special structure of the integral equations. By continuity of the integration kernel nearby rows of $A$ are similar. Consequently, averaging blocks of rows will be approximately equivalent to considering a lower dimensional discretisation of the integral equation. The singular value decomposition of the discretised problem tends to the one of the infinite-dimensional problem as the dimension of the discretisation goes to infinity and therefore increasing the discretisation dimension will be approximately equivalent to adding additional (high-frequency) singular vectors (of the ideal problem). In this sense the setting is related to the one investigated rigorously in this article. We formulate the approach precisely. We set $D:=2^{l_{\rm max}}$ with $l_{\rm max}:=12$ and $m_l:=2^l$ for $l=1,...,D$.  Then our discretisation operator becomes 

$$P_{m_l}:=\frac{1}{\sqrt{m_l}} \texttt{ eye}(D/m_l)\otimes\texttt{ ones}(1,m_l) $$

where $A\otimes B$ is the Kronecker-product of $A$ and $B$ and where $\texttt{ones}(m_l)_{j}=1$ for $j=1,...,m_l$ is an $m_l$ dimensional vector full of ones and $\texttt{eye}(m_l)_{i,j}=\delta_{ij}$ for $i,j=1,...,D/m_l$ is the $D/m_l\times D/m_l$ identity matrix. For example, setting $D=8$ and $m_l=4$ we obtain 

$$P_2 = \frac{1}{\sqrt{2}}\begin{pmatrix} 1 &1&&&&&&\\&&1&1&&&&\\ &&&&1&1&&\\ &&&&&&1&1\end{pmatrix}\in\R^{4\times 8}.$$

Thus $P_{m_l}A$ is an averaged lower dimensional version of $A$. In every run we terminate the iteration after at most $5*10^8$ iterations, if the stopping criteria is not fulfilled before.

The norm of the right hand side $y^\dagger$ is of order $1$ in all cases, thus it holds that $D \delta^2$ is also of order $1$ for the middle noise level $\delta=10^{-2}$. We therefore have three different scenarios, one of large noise ($\delta=1$) where the overall noise dominates the data, one of average noise ($\delta=10^{-2}$) where noise and data norm are on par, and one of small noise ($\delta=10^{-4}$) where the noise is much smaller than the data norm. 

 As a measurement of the accuracy we calculate the sample mean $e_{\cdot}$ of $\|x_{k^\delta_{\cdot}}-x^\dagger\|$ with $\cdot \in \{{\rm dp}, {\rm es}\}$ for 100 independent runs. We chose $\tau=1.5$, similar as in the recent numerical survey from Werner \cite{werner2018adaptivity}. The mean error for the early stopping discrepancy principle happens to be very large so we also calculate the median $\overline{e}_{\rm es}$ of  $\|x_{k^\delta_{\rm es}}^\delta-x^\dagger\|$ to reduce the impact of outliers. We  mention that the early stopping discrepancy principle often reaches the maximum number of iterations. The error of the modified discrepancy principle turns out to be very concentrated, so we do not depict its median.

Since the main motivation of the sequential early stopping discrepancy principle are its low computational costs, we estimate the numerical complexity of both methods. The main operation in each iteration step of the Landweber method is a matrix-vector multiplication, whose complexity is approximately the size of the dimension of the matrix. Therefore, we set $c_{\rm dp}:=\sum_{l=1}^{l_{\max}} {\rm mean}\left(\kdp(m_l)\right) m_l^2$ to be the complexity of the modified discrepancy principle and $c_{\rm es}:={\rm mean}(k^\delta_{\rm es}) D^2$ and $\overline{c}_{\rm es}:={\rm median}(k^\delta_{\rm es}) D^2$ to be the complexities of the early stopping discrepancy principle. Moreover we show the mean of $\mdp$, i.e., the discretisation dimension chosen by the modified discrepancy principle. In accordance with Corollary \ref{cor3} we use $1/2$ instead of $1/2\tau e^4$ in the \texttt{while}-loop of Algorithm 1. In order to illustrate the impact of this we compare the results to the one where we do not apply the \texttt{while}-loop at all, i.e., for the choices \eqref{int:pc} and \eqref{int:pc1} which are indicated by a bar over the respective variable.
 We display the results in Table \ref{tab:er-phillips}-\ref{tab:er-heat}. 

Interestingly, the complexity of the modified discrepancy principle is considerably lower than the complexity of the early stopping discrepancy principle. This is also true if we consider only the median of the latter and is connected to the slow convergence of the Landweber method and the considerably large fudge parameter $\tau=1.5$. 

In terms of accuracy the mean error of the modified discrepancy principle and the median error of the early stopping discrepancy principle are approximately on par for all noise levels. The mean error of the early stopping discrepancy principle is only comparable for the smallest tnoise level and otherwise of different order.

 Moreover, we see that the \texttt{while}-loop in Algorithm 1 is not affecting the results here very much.

All in all, this small numerical study indicates that the modified discrepancy principle is a computational attractive and efficient method to solve ill-posed integral equations.

There are several open points. E.g., in order to decrease the computational complexity further it would make sense to reduce also the dimension of the data space, e.g., through averaging of the columns of the matrix. Another interesting subject is the role of the fudge parameter $\tau$ in between the early stopping and the modified discrepancy principle.

\begin{table}[hbt!]
	\centering
	\caption{Comparison of modified and early stopping discrepancy principle for \texttt{phillips}.\label{tab:er-phillips}}
	\setlength{\tabcolsep}{4pt}
	\begin{tabular}{c|ccccccccc|}
		\toprule
		$\delta$	& $e_{\rm dp}$ & $\overline{e}_{\rm dp}$ &  $e_{\rm es}$ &  $\overline{e}_{\rm es}$& $c_{\rm dp}$ &$c_{\rm es}$ &$\overline{c}_{\rm es}$ &$m_{\rm dp}$ & $\overline{m}_{dp}^\delta$\\
		e0    & 2.7e0 & 2.7e0      & 7.3e3 &3e0 & 2.2e7 &2.7e14 &2.5e8 &6.8e0& 2.3e0\\
		e-2  & 2.9e-1 &3.0e-1   & 7.3e1& 4.8e-1 & 7.0e7 &2.7e14&7.7e9 &1.1e1&8.1e0\\
		e-4     & 4.6e-2 &4.6e-2     & 8.8e-1& 1.2e-1 & 2.7e9 &3.1e14& 4.5e12 &1.6e1&1.6e1\\		\bottomrule		
	\end{tabular}
\end{table}

\begin{table}[hbt!]
	\centering
	\caption{Comparison of modified and early stopping discrepancy principle for \texttt{deriv2}.\label{tab:er-deriv2}}
	\setlength{\tabcolsep}{4pt}
	\begin{tabular}{c|ccccccccc|}
		\toprule
		$\delta$	& $e_{\rm dp}$ & $\overline{e}_{\rm dp}$ &  $e_{\rm es}$ &  $\overline{e}_{\rm es}$& $c_{\rm dp}$ &$c_{\rm es}$ &$\overline{c}_{\rm es}$ &$\mdp$ & $\overline{m}_{dp}^\delta$\\
		e0    & 8.8e-1 & 9.3e-1      & 6.5e3 &5.8e-1 & 2.2e7 &1.1e14 &1.7e8 &2.9e0& 2.1e0\\
		e-2  & 2.9e-1 &2.9e-1   & 6.4e1& 3.5e-1 & 2.4e7&1.3e14&2.3e9 &3.8e0&2.7e0\\
		e-4     & 1.4e-1 &1.4e-1     & 8.2e-1& 1.5e-1 & 3.7e9 &1.3e14& 6.9e11 &1.6e1&1.5e1\\			\bottomrule	
	\end{tabular}
\end{table}

\begin{table}[hbt!]
	\centering
	\caption{Comparison of modified and early stopping discrepancy principle for \texttt{gravity}.\label{tab:er-gravity}}
	\setlength{\tabcolsep}{4pt}
	\begin{tabular}{c|ccccccccc|}
		\toprule
		$\delta$	& $e_{\rm dp}$ & $\overline{e}_{\rm dp}$ &  $e_{\rm es}$ &  $\overline{e}_{\rm es}$& $c_{\rm dp}$ &$c_{\rm es}$ &$\overline{c}_{\rm es}$ &$\mdp$ & $\overline{m}_{dp}^\delta$\\
		e0    & 1.1e1 & 1.4e1      & 3.0e3 &1.9e1 & 2.4e7 &3.7e14 &2.3e9 &7.6e0& 2.5e0\\
		e-2  & 1.9e0 &1.9e0   & 2.7e1& 2.6e0 & 6.4e8 &3.3e14&1.0e10 &1.0e1&8.6e0\\
		e-4     & 4.3e-1 &4.3e-1     & 4.9e-1& 4.3e-1 & 1.3e11 &2.9e14& 3.3e12 &1.6e1&1.6e1\\			\bottomrule	
	\end{tabular}
\end{table}

\begin{table}[hbt!]
	\centering
	\caption{Comparison of modified and early stopping discrepancy principle for \texttt{heat}.\label{tab:er-heat}}
	\setlength{\tabcolsep}{4pt}
	\begin{tabular}{c|ccccccccc|}
		\toprule
		$\delta$	& $e_{\rm dp}$ & $\overline{e}_{\rm dp}$ &  $e_{\rm es}$ &  $\overline{e}_{\rm es}$& $c_{\rm dp}$ &$c_{\rm es}$ &$\overline{c}_{\rm es}$ &$\mdp$ & $\overline{m}_{dp}^\delta$\\
		e0    & 1.3e1 & 1.3e1      & 4.3e3 &1.6e1 & 2.2e7 &1.6e14 &5.9e8 &3.0e0& 2.1e0\\
		e-2  & 3.2e0 &7.4e0   & 5.8e1& 3.8e0 & 8.4e8 &2.2e14&2.3e10 &1.6e1&8.3e0\\
		e-4     & 6.3e-1 &8.1e-1     & 8.1e-1& 4.2e-1 & 2.5e11 &2.0e14& 1.8e12 &4.7e1&2.4e1\\			\bottomrule	
	\end{tabular}
\end{table}

\begin{funding}
	Funded  by  the  Deutsche  Forschungsgemeinschaft under Germany's Excellence Strategy - GZ 2047/1, Projekt-ID 390685813.
\end{funding}

\bibliographystyle{imsart-number} 
\bibliography{references}    

\end{document}